\documentclass[leqno,12pt]{article}

\usepackage{latexsym,amsmath,amsthm,amssymb,amsfonts,amscd}
\usepackage{euscript}
\usepackage{graphicx}
\usepackage{flafter}
\usepackage{pstricks}

\usepackage{mathrsfs} 

\evensidemargin\oddsidemargin

\begin{document}

\baselineskip=18pt
\setcounter{page}{1}

\renewcommand{\theequation}{\thesection.\arabic{equation}}
\newtheorem{theorem}{Theorem}[section]
\newtheorem{lemma}[theorem]{Lemma}
\newtheorem{proposition}[theorem]{Proposition}
\newtheorem{corollary}[theorem]{Corollary}
\newtheorem{remark}[theorem]{Remark}
\newtheorem{fact}[theorem]{Fact}
\newtheorem{problem}[theorem]{Problem}

\newcommand{\eqnsection}{
\renewcommand{\theequation}{\thesection.\arabic{equation}}
    \makeatletter
    \csname  @addtoreset\endcsname{equation}{section}
    \makeatother}
\eqnsection


\def\r{{\mathbb R}}
\def\e{{\mathbb E}}
\def\p{{\mathbb P}}
\def\q{{\mathbb Q}}
\def\P{{\bf P}}
\def\E{{\bf E}}
\def\Q{{\bf Q}}
\def\z{{\mathbb Z}}
\def\W{{\mathbb W}}
\def\T{{\mathbb T}}

\def\N{{\mathbb N}}

\def\sD{\mathscr{D}}
\def\sL{\mathscr{L}}
\def\sQ{\mathscr{Q}}
\def\sP{\mathscr{P}}
\def\sN{\mathscr{N}}
\def\sH{\mathscr{H}}
\def\sF{\mathscr{F}}
\def\sR{\mathscr{R}}

\newcommand{\be}{\begin{equation}}
\newcommand{\ee}{\end{equation}}

\def\ee{\mathrm{e}}
\def\d{\, \mathrm{d}}

\newcommand{\Npop}[2]{\sN^{(#1)}_{#2}}
\newcommand{\tauu}[2]{\tau^{(#1)}_{#2}}



\title{Branching Brownian motion seen from its tip}

\author{
E. A\"id\'ekon\thanks{Department of Mathematics and Computer Science, Technische Universiteit Eindhoven, The Netherlands. email: {\tt elie.aidekon@gmail.com} } \
J. Berestycki\thanks{Laboratoire de Probabilit\'es et Mod\`eles Al\'eatoires, CNRS UMR 7599, UPMC Universit\'e Paris 6, Case courrier 188, 4, Place Jussieu 75252 Paris Cedex 05.  email: {\tt julien.berestycki@upmc.fr}} \
\'{E}. Brunet\thanks{ Laboratoire de Physique Statistique,	\'Ecole Normale Sup\'erieure, UPMC Universit\'e Paris 6, Universit\'e Paris Diderot, CNRS, 24 rue Lhomond, 75005 Paris, France. email: {\tt eric.brunet@lps.ens.fr}} \
Z. Shi\thanks{Laboratoire de Probabilit\'es et Mod\`eles Al\'eatoires, CNRS UMR 7599, UPMC Universit\'e Paris 6, Case courrier 188, 4, Place Jussieu 75252 Paris Cedex 05.  email: {\tt zhan.shi@upmc.fr}} 
}

\maketitle
\vspace{0.2in}

\begin{abstract}

It has been conjectured since the work of Lalley and Sellke \cite{lalley-sellke} that the branching Brownian motion seen from its tip (e.g. from its rightmost particle) converges to an invariant point process. Very recently, it emerged that this can be proved in several different ways (see e.g.\@ Brunet and Derrida \cite{brunet derrida}, Arguin et al. \cite{ABK2,ABK3}).  The structure of this extremal point process turns out to be a  Poisson point process with exponential intensity in which each atom has been decorated by an independent copy of an auxiliary point process. The main goal of the present work is to give a complete description of the limit object via an explicit construction of this decoration point process. 
Another proof and description has been obtained independently by Arguin et al. \cite{ABK3}.  
\end{abstract}

\section{Introduction}
\label{s:problematique}

Branching Brownian motion is the subject of a large literature that one can trace back at least to \cite{INW}. The connection of this probabilistic model with the well-known F-KPP equation has in particular attracted much interest from both the probabilistic and the analytic side starting with the seminal studies of McKean \cite{mckean}, Bramson \cite{bramson83}, Lalley and Sellke \cite{lalley-sellke}, Chauvin and Rouault \cite{chauvinrouault}  and more recently with works by  Harris \cite{harris},  Kyprianou \cite{harris} and Harris, Harris and Kyprianou  \cite{hhk}.

In the present work we  consider a continuous-time branching Brownian motion with quadratic branching mechanism:  the system starts with a single particle at the origin which follows a Brownian motion with drift $\varrho$ and variance  $\sigma^2>0$. After an exponential time with parameter $\lambda>0$ the particle splits into two new particles which each start a new independent copy of the same process started from it place of birth. Each of them thus moves according to a Brownian motion with drift $\varrho$ and variance $\sigma^2>0$ and splits into two after an exponential time with parameter $\lambda>0$ and so on.

We write $X_1(t) \le \ldots \le X_{N(t)}$ for the positions of the particles of the branching Brownian motion alive at time $t$ enumerated from left to right (where $N(t)$ is the number of particles alive at time $t$). The corresponding random point measure is denoted by 
$$\mathscr{N}(t) := \sum_{i=1,\ldots, N(t)} \delta_{X_i(t)}. $$
We will work under conditions on $\lambda, \varrho, \sigma^2$ which ensure that for all $t>0$,
\begin{equation}
    \E\Big[ \sum_{i =1, \ldots , N(t)} \ee^{-X_i(t)}\Big] =1, \qquad
    \E\Big[ \sum_{i =1, \ldots , N(t)} X_i(t) \ee^{-X_i(t)}\Big] =0.
    \label{cond-hab}
\end{equation}
\noindent Since $\E(N(t))=\ee^{\lambda t}$, for any measurable function $F$ and each $t>0$,
$$
\E\Big[ \sum_{i =1, \ldots , N(t)} F(X_{i,t}(s), \, s\in [0, \, t])\Big]
=
\ee^{\lambda t}\, \E\Big[ F(\sigma B_s + \varrho s, \, s\in [0, \, t])\Big] ,
$$

\noindent where, for each $i \in \{1,\ldots, N(t)\}$ we let $X_{i,t}(s), \, s\in [0, \, t]$ be the position, at time $s$, of the unique ancestor of $X_i(t)$  and  $B$ is a standard Brownian motion.
Thus the equations (\ref{cond-hab}) become $\varrho = \lambda + {\sigma^2\over 2}$ and $\varrho = \sigma^2$. Hence the usual conditions amount to supposing $\varrho = \sigma^2= 2\lambda$. {\bf In this paper we always assume $\lambda=1$, $\varrho=2$ and $\sigma=\sqrt2$.} The choice of a binary branching is arbitrary. Our results certainly hold true for a more general class of branching mechanisms, e.g. when the law of the number of offsprings is bounded or has finite second moment. For the sake of clarity we only consider the simple case of binary branching which already contains the full phenomenology.   

\medskip

The position $X_{N(t)}(t)$ of the rightmost  particle of the branching Brownian motion  has been much studied (see \cite{mckean, bramson78, bramson83, lalley-sellke}). In these classical works, the authors usually assume that $\varrho=0, \lambda= \sigma=1 .$ We recall some of their results adapted to our normalization. In particular, instead of the rightmost particle
we prefer to work with
the position $X_1(t)$ of the leftmost particle.

Bramson \cite{bramson83} shows that there exists a constant $C_B \in \r$ and a real valued random variable $W$ such that
\begin{equation}
    X_1(t) -  m_t
    \;\; {\buildrel \text{law} \over \to} \;\;
    W, \qquad t \to \infty,
    \label{bramson}
\end{equation}
\noindent where
\begin{equation}
    m_t := {3  \over 2} \log t + C_B
    \label{bramson-constant}
\end{equation}
and furthermore the distribution function $\P(W\le x) =w(x)$ is a solution to the critical F-KPP travelling wave equation
$$  w'' + 2w'+w(w-1)=0.$$

Lalley's and Sellke's paper
\cite{lalley-sellke} can be seen as the real starting point of the present work. Realizing that the convergence (\ref{bramson}) cannot hold in an ergodic sense, they prove the following result. Define
\begin{equation}\label{E:def de Z}
Z(t) : = \sum_{i=1,\ldots, N(t)} X_i(t) \ee^{-X_i(t)}.
\end{equation}
We know that $\E(Z(t)) = 0$ by (\ref{cond-hab}) and it is not hard to see that $(Z(t), t\ge 0)$ is in fact a martingale (the so-called {\it derivative martingale}). It can be shown that
\begin{equation}
    Z :=\lim_{t \to \infty} Z(t)
    \label{E:def de Z1}
\end{equation}

\noindent exists, is finite and strictly positive with probability 1.
The main result of Lalley's and Sellke's paper is then that $\exists C>0$
such that
$$
\lim_{s \to \infty} \lim_{t \to \infty} \P (X_1(t+s) - m_{t+s} \ge x | \sF_s ) = \exp \left(  - CZ\ee^{ x}\right)
$$
where $\sF_t$ is the natural filtration of the branching Brownian motion.
As a consequence,
\begin{equation}
    \P(W  \le x)
    \;\sim \;
    C \, |x| \ee^{ x},
    \qquad x\to -\infty .
    \label{lalley-sellke}
\end{equation}

Since conditionally on $Z$
the function
$y \mapsto \exp \left(  - CZ \ee^{ y} \right) =  \exp \left(  -  \ee^{ y +\log (CZ) } \right)$ is the distribution function of minus a Gumbel random variable centered on $-\log (CZ),$ this suggests the following picture which is conjectured by Lalley and Sellke for the front of
branching Brownian motion.
The random variable $X_1(t) - m_t$ converges in distribution and its limit is the sum of two terms. The first one is $-\log (CZ),$ which depends on the limit of the derivative martingale, while the second term is simply minus a Gumbel random variable. Brunet and Derrida \cite{brunet derrida} interpret this as a random delay (which builds up early in the process and settles down to some value) and a fluctuation term around this position.

In the last section of \cite{lalley-sellke}, the authors conjecture that more generally, the point measure of particle positions relative to $ m_t-\log (CZ) $
$$
\bar \sN(t) :=  \sum_{i=1,\ldots, N(t)}\delta_{X_i(t) -  m_t +\log (CZ)}
$$
converges to a stationary distribution. 

In the present work we prove that $\bar \sN(t)$ converges to a stationary distribution which we describe precisely. We show that the structure of this limit point measure is a {\it decorated } Poisson point measure, i.e., a Poisson point measure on the real line where each atom is replaced by an independent copy of a certain point measure shifted by the position of the atom.
Another proof and description has been obtained independently by Arguin et al. \cite{ABK3} (see Section \ref{S:discut}).

\section{Main results}\label{S:main}


Throughout the paper, all point measures are, as in the setting of Kallenberg~\cite{kallenberg}, considered as elements of the space $\mathcal{M}$ of Radon measures on $\mathbb{R}$ equipped with the vague topology, that is, we say that $\mu_n$ converges in distribution to $\mu$ if and only if $\int f \d\mu_n\to \int f\d\mu$ for any real continuous function $f$ with compact support. By Theorem 4.2~(iii) p.~32 of \cite{kallenberg}, it is equivalent to say that $(\mu_n(A_j),\,1\le j\le k)$ converges in distribution to $(\mu(A_j),\,1\le j\le k)$ for any intervals $(A_j,\,1\le j\le k)$. The space $C(\r_+, \, \r)$ (or sometimes, $C([0, \, t], \, \r)$) is endowed with topology of uniform convergence on compact sets.  If $F$ is a function on $C(\r_+, \, \r)$, then for any continuous function $(Z_s, \, s\in [0, \, t])$, we define $F(Z_s, \, s\in [0, \, t])$ as $F(\widetilde{Z}_s, \, s\ge 0)$, with $\widetilde{Z}_s := Z_{\min\{ s, \, t\}}$. 

We now introduce two point measures which are the main focus of this work. First, consider the point measure of the particles seen from $m_t - \log (CZ)$ and enumerated from the leftmost:
$$
\bar \sN (t) = \sN(t) -  m_t + \log (CZ)
=
\sum_{i=1, \ldots, N(t)} \delta_{X_i(t) - m_t + \log (CZ) }.
$$

We will also sometimes want to consider the particles as seen from the leftmost
$$
\sN'(t)
:=
\sum_{i=1, \ldots, N(t)} \delta_{X_i(t) - X_1(t) }
$$

\begin{theorem}
\label{t:main}
As $t \to \infty$ the pair $\{ \bar \sN(t) , Z(t) \} $ converges jointly in distribution to  $ \{ \sL , Z \}$ where $Z$ is as in (\ref{E:def de Z1}),  $\sL$ and $Z$ are independent and $\sL$ is  obtained as follows.
\begin{itemize}
\item[(i)] Define $\sP$ a Poisson point measure on $\r$, with intensity measure  $  \ee^{ x}\d x.$
\item[(ii)] For each atom $x$ of $\sP$, we attach a point measure $\sQ^{(x)}$ where $\sQ^{(x)}$ are
independent copies
of a certain point measure $\sQ$.
\item[(iii)]  $\mathscr{L}$ is then the point measure corresponding to the sum of all $x+\mathscr{Q}^{(x)}$, i.e., $$\mathscr{L} := \sum_{x \in \sP} \sum_{y \in \sQ^{(x)}}\delta_{x+y}$$
where $x\in \sP$ means ``$x$ is an atom of $\sP$''.
\end{itemize}
\end{theorem}

Since the leftmost atom of $\mathscr{P}$ has the Gumbel distribution, this implies that the Gumbel distribution is the weak limit of $X_1(t) - m_t + \log (CZ)$. The following corollary, concerning the point measure seen from the leftmost position, contains strictly less information than the theorem.
\begin{corollary}
As $t \to \infty$ the point measure $\sN'(t)$ converges in distribution to the point measure $\sL'$  obtained by replacing the Poisson point measure $\sP$ in step (i) above by $\sP'$ described in step (i)' below:
\begin{itemize}
\item[(i)'] Let ${\bf e}$ be a standard exponential random variable. Conditionally on ${\bf e},$ define $\sP'$ to be a Poisson point measure on $\r_+$, with intensity measure ${\bf e} \ee^{x} {\bf 1}_{\r_+}(x) \d x$ to which we add an atom in $0$.
\end{itemize}
The decoration point measure $\sQ(x)$ remains the same.
\end{corollary}

\medskip

The variable $Z$ is not $\sF_t$-measurable, and in this sense Theorem \ref{t:main} is a conditional statement. However, it is clear that if one replaces $\bar \sN(t)$ by $$ \hat \sN(t) := \sN(t) -  m_t + \log (CZ(t))
=
\sum_{i=1, \ldots, N(t)} \delta_{X_i(t) - m_t + \log (CZ(t)) }$$
which is $\sF_t$-measurable, then the same result still holds.

\medskip

Theorem \ref{t:main} above should not be considered a new result when the {\it decoration} point measure $\sQ$ is not specified.  Indeed, the convergence to a limiting point process was already implicit in the results of Brunet and Derrida \cite{brunet-derrida} and is also proved independently in \cite{ABK3} by Arguin et al. See Section \ref{S:discut} for a detailed discussion.

\medskip

We next give a precise description of  the {\it decoration} point measure $\sQ$ which is the main result of the present work. For each $s \le t$, recall that $X_{1,t}(s)$ is the position at time $s$ of the ancestor of $X_1(t)$, i.e., $s\mapsto X_{1,t}(s)$ is the path followed by the
leftmost particle
at time $t.$ We define
$$
Y_t(s) := X_{1,t}(t-s) - X_1(t) , \qquad s\in [0,t]
$$
the time reversed path back from the final position $X_1(t).$ 
Let us write $t\ge \tau_1(t)>\tau_2(t)>\ldots $ for the (finite number of) successive splitting times of branching along the trajectory $X_{1,t}(s), s\le t$ (enumerated backward). We define  $\mathscr{N}_i(t)$ to be the point measure corresponding to the set of all particles at time $t$ which have branched off from $X_{1,t}$ at time $\tau_i(t)$ relative to the final position $X_1(t)$ (see figure 1).
We will also need the notation $\tau_{i,j}(t)$ which is the time at which $X_i(t)$ and $X_j(t)$ share their most recent common ancestor. Observe that
$$
\mathscr{N}_i(t) = \sum_{j \le N(t) : \tau_{1,j}(t) = \tau_i(t)} \delta_{X_j(t) - X_1(t)} \ .
$$

\begin{figure}
\begin{center}
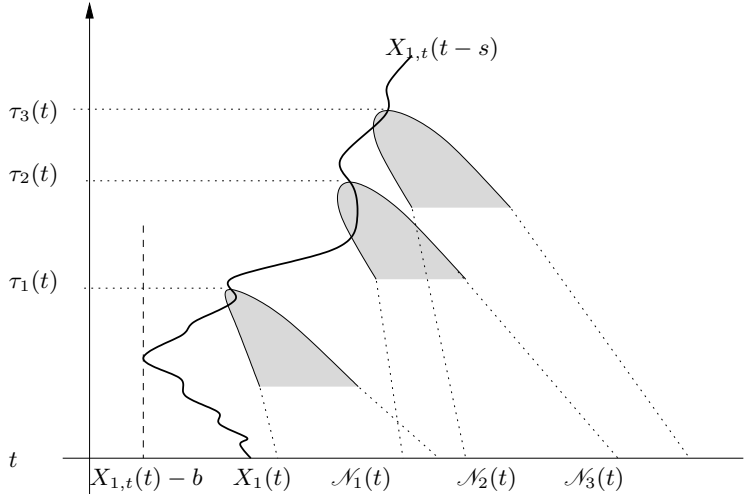
\caption{ $(Y, \sQ)$ is the limit of the path $s\mapsto X_{1,t}(t-s)-X_1(t) $ and of the points that have branched recently off from $X_{1,t}$.}
\label{F:Qstruct}
\end{center}
\end{figure}

We then define
$$
\sQ(t,\zeta) := \delta_0 +\sum_{i : \tau_i(t) >t-\zeta}\sN_i(t)
$$
i.e.,  the point measure of particles at time $t$ which have branched off $X_{1,t}(s)$ after time $t-\zeta$, including the  particle at $X_1(t)$ itself.

We will first show that $((Y_t(s), s \in [0,t]) , \sQ(t,\zeta))$ converges jointly in distribution  (by first letting $t \to \infty$ and then $\zeta \to \infty$) towards a limit $((Y(s), s\ge 0) , \sQ)$ where the second coordinate is our point measure $\sQ$ which is described by growing conditioned
branching Brownian motions
born at a certain rate on the path $Y.$ We first describe the limit $((Y(s), s\ge 0) , \sQ)$ and then we state the precise convergence result.

\medskip

The following family of processes indexed by a real parameter $b>0$ plays a key role in this description. Let $B:= (B_t, \, t\ge 0)$ be a standard Brownian motion and let $R:= (R_t, \, t\ge 0)$ be a three-dimensional Bessel process started from $R_0:=0$ and independent from $B.$ Let us define $T_b := \inf\{ t\ge 0: \, B_t =b\}$. For each $b >0$, we define the process $\Gamma^{(b)}$ as follows:
\begin{equation}
    \Gamma^{(b)}_s
    :=
    \begin{cases} B_s, &\text{ if $s\in [0, \, T_b]$}, \cr\cr
    b-R_{s-T_b}, &\text{ if $s\ge T_b$.} \cr \end{cases}
    \label{U}
\end{equation}

\begin{figure}[h]
\begin{center}
 \includegraphics{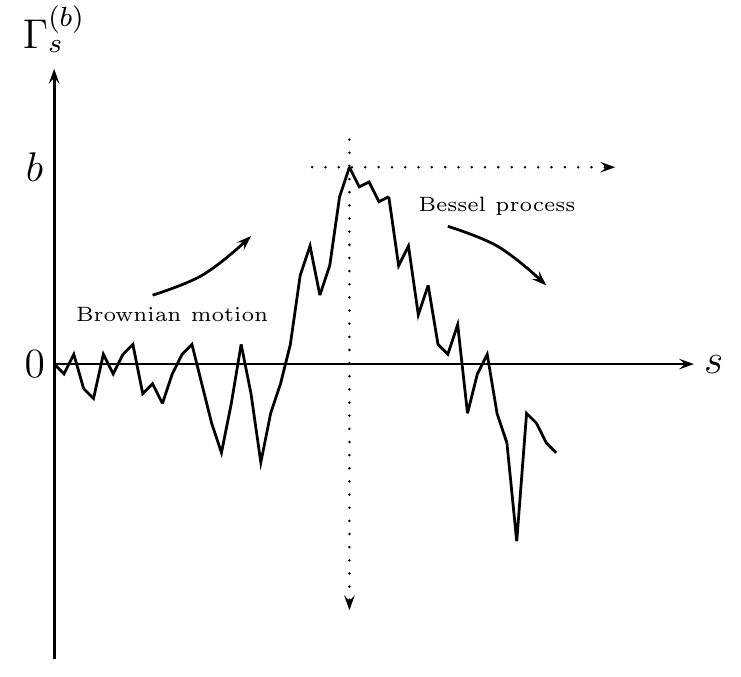}
 \end{center}
 \caption{the process $\Gamma^{(b)}$}
\end{figure}

Let us define
$$
G_t(x) :=  \P_0(X_1(t) \le x) =\P_{-x} (X_1(t) \le 0)
$$
the probability of presence to the left of $x$ at time $t$, where we write $\P_x$ for the law of the
branching Brownian motion
started from one particle at $x.$ Hence, by (\ref{bramson}) we see that 
$G_t (x+m_t) \to \P(W \le x)$.

We can now describe the law of the backward path $Y:$ for any measurable set $A$ of $C(\r_+, \, \r)$ and $b \ge 0$,
$$
\P (Y\in A , -\inf_{s \ge 0}Y(s) \in db ) = {1\over c_1} \E \left[  \ee^{-2  \int_0^\infty G_v (\sigma \Gamma^{(b)}_v) \d v}  \mathbf{1}_{ -\sigma \Gamma^{(b)} \in A} \right],
$$
where 
\begin{equation*}
     c_1
     :=
     \int_0^\infty
     \E \Big[ \ee^{-2 \int_0^\infty G_v( \sigma \Gamma_v^{(b)})\d v}  \Big]
     \d b
 \end{equation*}
(observe that by equation \eqref{E:c1} this constant is finite).

Observe that 
$-\inf_{s \ge 0}Y(s)$ is a random variable with values in
$(0, \, \infty)$
whose density is given by $$\P( -\inf_{s \ge 0}Y(s) \in \!\d b) =  {1\over c_1}\E \Big[ \ee^{-2 \int_0^\infty G_v( \sigma \Gamma_v^{(b)})\d v}  \Big] \d b.
$$

Now, conditionally on the path $Y,$ we let $\pi$ be a Poisson point process on $[0,\infty)$ with intensity $2  \big(1 - G_t(-Y(t))\big) \d t  = 2  \big(1 - \P_{Y(t)} (X_1(t) <0) )\big) \d t$. For each point $t\in \pi$ start an independent
branching Brownian motion
$(\sN^*_{Y(t)}(u), u\ge 0)$ at position $Y(t)$ conditioned to have $\min \sN^*_{Y(t)}(t) >0$.\footnote{By convention, for a point measure $\sN$, $\min \sN$ is the infimum of the support of $\sN$.}  Then define $\sQ:= \delta_0+ \sum_{t \in \pi} \sN^*_{Y(t)}(t).$

\begin{theorem}\label{T: structure de Q}

The following convergence holds jointly in distribution:
$$
\lim_{\zeta \to \infty} \lim_{t \to \infty}
((Y_t(s), s\in [0,t]) , \,
\sQ(t,\zeta), \, X_{1}(t)-m_t)
=
((Y(s), s\ge 0) , \, \sQ, \, W) ,
$$
where the random variable $W$ is independent of the pair $((Y(s),s\ge 0), \, \sQ)$, and $\sQ$ is the point measure which appears in Theorem \ref{t:main}.
\end{theorem}
Observe that the parameter $\zeta$ only matters for the decoration point measure in the second coordinate.
\\

The following Theorem \ref{T: path of X_1}
characterizes the joint distribution of the path $s \mapsto X_{1,t}(s)$ that the particle which is the leftmost at time $t$ has followed, of the point measures of the particles to its right, and of the times at which
these particles
have split in the past, all in terms of a Brownian motion functional.
The proof  borrows some ideas from \cite{elie} but is more intuitive in the present setting of branching Brownian motion. Moreover, it also serves as a first step in the 
(much) more involved proof of Theorem \ref{T: structure de Q}
in Section \ref{S: preuve de 2.3}.

For any positive measurable functional $F : C([0,t],\r) \mapsto \r_+$  and any positive measurable function $f:[0, \, t] \to \r_+$, for $n \in \N, (\alpha_1, \ldots, \alpha_n) \in \r_+^n$ and $A_1,\ldots, A_n$ a collection of Borel subsets of $\r_+$  define
$$
I(t) := \E\Big\{ F(X_{1,t}(s), \, s\in [0, \, t])\,
    \exp\Big( -
    \sum_i
    f(t-\tau_i(t))\,
    \sum_{j=1}^n \alpha_j
     \int_{ A_j} \d \mathscr{N}_i(t)  \Big)
    \Big\} ,
$$
where for a point measure $\sN$ and a set $A$ we write $\int_A \d \sN$ in place of $ \sN(A).$

For each $r\ge 0$ and every $x\in \r$ recall that $G_r(x)
= \P \{ X_1(r) \le x\},$ and further define
\begin{eqnarray*}
\overline{G}_r^{(f)}(x)
&:=& \E \Big[ \ee^{-f(r) \sum_{j=1}^n \alpha_j [\int_{x+ A_j} \d \mathscr{N}(r) ] } \, {\bf 1}_{ \{ X_1(r) \ge x\} } \Big].
\end{eqnarray*}
Hence, when $f \equiv 0$ we have $\overline{G}_r^{(f)}(x) = 1-G_r(x)$.

\begin{theorem}\label{T: path of X_1}

We have
\begin{equation}
    I(t)  =\E \Big[ \ee^{\sigma B_t} \, F(\sigma B_s, \, s\in [0, \, t]) \,
    \ee^{- 2  \int_0^t [1-\overline{G}_{t-s}^{(f)}(\sigma B_t-\sigma B_s)] \d s} \Big],
    \label{caracterisation-epinale1}
\end{equation}
where $B$ in the expectation above is a standard Brownian motion. In particular, the path $(s \mapsto X_{1,t}(s), 0\le s\le t)$ is a standard Brownian motion in a potential:
\begin{equation}\label{E: bm in a potential}
\E\Big[  F(X_{1,t}(s), \, s\in [0, \, t])  \Big]=\E \Big[ \ee^{\sigma B_t} \, F(\sigma B_s, \, s\in [0, \, t]) \,
    \ee^{- 2  \int_0^t G_{t-s}(\sigma B_t-\sigma B_s) \d s} \Big].
\end{equation}
\end{theorem}
This result, which can be seen as a Feynman-Kac representation formula is hardly surprising and is reminiscent of the approach in Bramson's work.

\bigskip

In addition to this ``Brownian motion in a potential" description we also present some properties of a typical path $(X_{1,t}(s), \, s\in [0, \, t]).$  
Let us fix a constant $\eta >0$ (that we will take large enough in a moment). For $t \ge 1$ and $x>0$, we define the good event $A_{t}(x,\eta)$ by
\begin{equation}\label{def:goodevent}
A_{t}(x,\eta) := E_1(x,\eta)\cap E_2(x,\eta) \cap E_3(x,\eta) 
\end{equation}

\begin{figure}\label{F:path}
\begin{center}
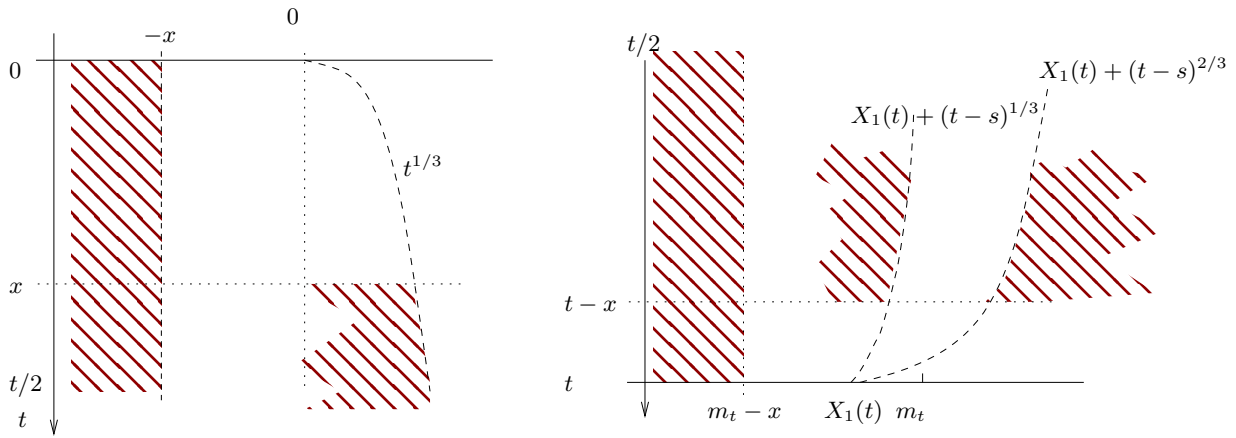
\caption{The events $E_1(x,\eta),E_2(x,\eta) $ and $E_3(x,\eta)$ together are the event that the paths of particles ending within distance $\eta$ of $m_t$ avoid all the dashed regions. }
\end{center}
\end{figure}


\noindent where the events $E_i$ (see figure \ref{F:path}) are defined by
\begin{eqnarray*}
E_1(x,\eta) &:=& \Big\{ \forall i \text{ s.t. } |X_i(t)-m_t| <\eta , \ \min_{s\in [0,t]} X_{i,t}(s) \ge -x,\, \min_{s\in [t/2,t]} X_{i,t}(s) \ge m_t -x \Big\},\\
E_2(x,\eta) &:=& \Big\{ \forall i \text{ s.t. } |X_i(t)- m_t| <\eta , \forall s\in[x,{t\over 2}], X_{i,t}(s) \ge s^{1/3} \Big\},\\
E_3(x,\eta) &:=& \Big\{ \forall i \text{ s.t. } |X_i(t)- m_t| <\eta , \forall s\in [x,{t\over 2}],\, X_{i,t}(t-s)-X_i(t) \in [s^{1/3},s^{2/3}]\Big\}.
\end{eqnarray*}

\noindent We will show that the event $A_{t}(x,\eta)$ happens with high probability, the reason being that $s\mapsto X_{1,t}(s)$ looks very much like a Brownian excursion over the curve $s\rightarrow m_s$. We observe that the events $E_i$ depend on $t$ but we omit to write the dependency for sake of brevity.  

\begin{proposition}[Arguin, Bovier and Kistler \cite{ABK1}]\label{T: prop de X_1}
Let $\eta>0$. For any $\varepsilon>0$, there exists $x>0$ large enough such that $\P(A_t(x,\eta)) \ge 1-\varepsilon$ for $t$ large enough.
\end{proposition}

Observe in particular, that since $\P(|X_{1,t}(t) -m_t | >\eta) \to 0$ when $\eta \to \infty$ we know that for $\eta$ and $x$ large enough, the path $s \mapsto X_{1,t}(s)$ has the properties described in the event $E_1,E_2,E_3$ with arbitrary high probability. Here the exponents $1/3$ and $2/3$ have been chosen arbitrarily in the sense that one could replace them with $1/2 \pm \varepsilon$ for any $0<\varepsilon <1/2$.

\medskip

The rest of this paper is organized as follows. Section \ref{S:discut} is devoted to discussions on related results. The main goal of the paper is to prove Theorem \ref{T: structure de Q}, which is also the hardest. We start by proving Theorem \ref{T: path of X_1} in Section \ref{s:proof-thm2.4} which is much easier, thus introducing some tools and ideas we will use throughout the paper. Next, in Section \ref{s:proof-prop2.5}, we prove Proposition \ref{T: prop de X_1} which gives us estimates on the localization of the path followed by the rightmost particle. Section \ref{S: preuve de 2.3} contains the main arguments for the proof of Theorem \ref{T: structure de Q}, and Sections \ref{s:rappels-Denisov-Imhof}, \ref{s:1er-appli-Denisov} and \ref{s:proof-lemma7.1} are devoted to technical intermediary steps. 

The proof of Theorem \ref{t:main} is given last in Section \ref{S:proof of thm main}. We show that by stopping particles when they first hit a certain position $k$ and then considering only their leftmost descendants one recovers a Poisson point measure of intensity $\ee^x \d x$ as $k \to \infty.$ Then,  we show that two particles near $m_t$   have separated in a branching event that was either very recent or near the very beginning of the process and we finally combine those two steps to complete the proof of Theorem  \ref{t:main}.


\section{Related results and discussion}\label{S:discut}


The goal of this section is to discuss the relevant literature and to give a brief account of the main differences and similarities between the present work and some related papers. 

\medskip

The description of the extremal point process of the branching Brownian motion is also the subject of \cite{brunet-derrida, brunet derrida} by Brunet and Derrida. There, using the McKean representation and Bramson's convergence result for the solutions of the F-KPP equation \cite{bramson83}, the authors show that the limit point process exists and has the superposability property. From there, using classical arguments (see for instance \cite{maillard}) it can then be shown that the only point processes having this property are those of the type ``decorated exponential Poisson point processes", proving in essence our Theorem \ref{t:main}. Recently, pursuing and adding to those ideas Arguin et al. have also shown the convergence of $\bar \sN(t)$ to a limiting point process with the superposability property (see \cite[Proposition 2.2 and Corollary 2.4]{ABK3}). Therefore, it is really Theorem \ref{T: structure de Q} --- the description of the decoration measure $\sQ$ --- which is the main contribution of the present work. Finally we mention that Madaule \cite{madaule} has proved the analogue of our Theorem \ref{t:main} for non-lattice branching random walks by using the recent result in \cite{elie} on the maximum of branching random walks.

\medskip

Most of the results presented here are identical or very closely related to those obtained independently by Arguin, Bovier and Kistler in a series of papers \cite{ABK1,ABK2,ABK3}. For reference we include here a brief description of their results, stated in the context of our normalization to ease comparison.

The main results of \cite{ABK1} concern the paths followed by the extremal particles and their genealogy. Our Proposition \ref{P: no branching at intermediate times} is the same result as Theorem 2.1 of \cite{ABK1} which says that particles near $m_t$ have either branched near time 0 or near time $t$. Theorems 2.2, 2.3 and 2.5 in \cite{ABK1}  concern the localization of paths of particles which end up near $m_t$ at time $t$. Arguin {\it et al.} show that at intermediary times $s$, with arbitrarily large probability, they lie between ${s \over t} m_t - (s\wedge(t-s))^\alpha$ and  ${s \over t} m_t - (s\wedge(t-s))^\beta$ for $0<\alpha <1/2 <\beta <1.$ This, of course, corresponds exactly to our Proposition \ref{T: prop de X_1}. Since their arguments rely essentially on many-to-one calculations and Bessel bridge estimates, the methods of proof are also very similar. We include the proofs of Propositions \ref{T: prop de X_1} and \ref{P: no branching at intermediate times} for the sake of self-containedness.

\medskip

In \cite{ABK2}, Arguin, Bovier and Kistler using the path localization argument obtained in \cite{ABK1} are able to show that if one only considers particles that have branched off from one another far enough into the past (the point process of maxima of the clusters), then it converges to a Poisson point process with exponential intensity (\cite{ABK2}, Theorem 2). This of course very closely resembles our Proposition \ref{P:convergence jointe vers P et Z}. Their proof relies on the convergence of Laplace functionals (for which a first Lalley-Sellke type representation is given) whereas we simply deduce this from the classical results about records of iid variables.

\medskip
In \cite{ABK3} a complete description of the extremal point process of the branching Brownian motion is given. There, they show that $\bar \sN(t)$ (actually in \cite{ABK3} the point process $\sN$ is centered by $m_t$ instead of $m_t - \log (CZ)$) converges in distribution to a limiting point process which is necessarily an exponential Poisson point process whose atoms are ``decorated" with iid point measures. They give a complete description of this decoration point measure as follows. Let $\sD(t) = \sum_{i=1}^\infty \delta_{X_i(t)-X_1(t)}$ which is a random point measure on $\r_+.$ Conditionally on the event $X_1(t)<0$ it converges in distribution to a limit $\sD$. Theorem 2.1 in \cite{ABK3} thus coincides with our Theorem \ref{t:main} via $\sQ=\sD.$ 

One of the key argument in \cite{ABK3} is to identify the limit extremal point process of the branching Brownian motion with the limit of an auxiliary point process. This auxiliary point process is constructed as follows. Let $(\eta_i, i \in \N)$ be the atoms of a Poisson point process on $\r_+$ with intensity $$a (x e^{b x} ) \d x$$ for some constants $a$ and $b$.
For each $i$, they start form $\eta_i$ an independent branching Brownian motion (with the same $\lambda,\sigma,\varrho$ parameters as the original one) and  call $\Pi(t)$ the point process of the position of all the particles of all the branching Brownian motions at time $t.$ Theorem 2.5 in \cite{ABK3} shows that $\lim_{t \to \infty} \Pi(t) = \lim_{t \to \infty} \bar \sN(t).$ This solves what Lalley and Sellke \cite{lalley-sellke} call the conjecture on the {\it standing wave of particles}. The proof is based on the analysis of Bramson \cite{bramson83} for the solution of the F-KPP equation with various initial conditions and the subsequent work of Lalley and Sellke \cite{lalley-sellke} and Chauvin and Rouault \cite{chauvin-rouault} which allows them to show convergences of Laplace type functionals of the extremal point process. 

In the present work we also prove the convergence of the extremal point process to a decorated exponential Poisson point process. Our main result, Theorem \ref{T: structure de Q}, gives a description of the decoration measure $\sQ$ which is very different from \cite{ABK3}. The methods we use are also different since we essentially rely on path localization and decomposition. It is our hope to exploit the description of $\sQ$ given in Theorem \ref{T: structure de Q} to prove a conjecture of Brunet and Derrida~\cite{brunet-derrida} concerning the asymptotic distribution of the extremal point measure $\mathscr{L}$.


\section{Proof of Theorem \ref{T: path of X_1}}
\label{s:proof-thm2.4}


We will use repeatedly the following approach which is known as the {\it spinal decomposition.}
The process
$$
M_t := \sum_{i \le N(t)} \ee^{-X_i(t)}, \qquad t\ge 0,
$$
is a so-called additive martingale, which is critical, not uniformly integrable and converges almost surely to 0. Let $\Q$ be the probability measure on  $\mathscr{F}_\infty$ such that, for each  $t\ge 0$,
$$
\Q_{|_{\mathscr{F}_t}} = M_t \bullet \P_{|_{\mathscr{F}_t}} \, .
$$
Following Chauvin and Rouault (\cite{chauvin-rouault}, Theorem 5), $\Q$ is the law of a branching diffusion with a particle behaving differently. More precisely, for each time $s\ge 0$ we let $\Xi_s \in \{1,\ldots, N(s)\}$ be the label of the distinguished particle (the process $(\Xi_s, \, s\in [0, \, t])$ is called the {\it spine}). The particle with label $\Xi_s$ at time $s$ branches at (accelerated) rate  $2$ and gives birth to normal branching Brownian motions (without spine)
with distribution $\P,$ whereas the process of the position of the spine $(X_{\Xi_s}(s), \, s\in [0, \, t])$ is a driftless Brownian motion of variance $\sigma^2=2$. Furthermore, for each $t\ge 0$ and each $i\le N(t)$,
$$
    \Q \{ \Xi_t =i \, | \, \mathscr{F}_t\}
    =
    {\ee^{- X_i(t)}\over M_t} \, .
$$
We use this principle repeatedly in the present work in the following manner. For each $i\le N(t)$ consider  $\Psi_i$  a random variable which is measurable in the filtration of the
branching Brownian motion
up to time $t$ (i.e., it is determined by the history of the process up to time $t$) and suppose that we wish to compute $\E_\P [ \sum_{i \le N(t)} \Psi_i  ].$ Then, thanks to the above, we have 
\begin{equation}
    \E_\P \Big[ \sum_{i \le N(t)} \Psi_i \Big]
    =
    \E_\Q \Big[ \frac1{M_t} \sum_{i \le N(t)}  \Psi_i \Big]
    =
    \E_\Q \Big[ \ee^{X_{\Xi_t}(t)} \Psi_{\Xi_t} \Big].
    \label{E:many to one version 1}
\end{equation}
We will refer to (\ref{E:many to one version 1}) as the many-to-one principle.


For any positive measurable function $F : C(\r_+,\r) \to \r_+,$ any positive measurable function $f:[0, \, t] \to \r_+$, $n \in \N, (\alpha_1, \ldots, \alpha_n) \in \r_+^n$ and $A_1,\ldots, A_n$ a collection of Borel subsets of $\r_+$  define
$$
I(t) := \E\Big\{ F(X_{1,t}(s), \, s\in [0, \, t])\,
    \exp\Big( -
    \sum_i
    f(t-\tau_i(t))\,
    \sum_{j=1}^n \alpha_j
     \int_{ A_j} \d \mathscr{N}_i(t)  \Big)
    \Big\} ,
$$

\noindent as in Section \ref{S:main}. Letting $X_{i,t}(s)$ be the position of the ancestor at time $s$ of the particle at $X_i(t)$ at time $t$, we have
$$
I(t)
=
\E\Big[ \sum_{i\le N(t)}
{\bf 1}_{ \{ i=1\} } \,
F(X_{i,t}(s), \, s\in [0, \, t]) \, \Lambda_i(t)\Big] ,
$$
\noindent with $\Lambda_i(t) := \exp\{ -
\sum_k f(t-\tau^{(i)}_k(t))\, \sum_{j=1}^n \alpha_j  [ \int_{ A_j} \d \mathscr{N}^{(i)}_k ] \}$ where the sequence of times $\tau^{(i)}_k(t)$ are the successive branching times along $X_{i,t}(s)$ enumerated backward from $t$,  and the point measures $\mathscr{N}^{(i)}_k$ are the particles which have branched off from $X_{i,t}(s)$ at time $\tau^{(i)}_k(t)$
$$
\mathscr{N}^{(i)}_k : =\sum_{\ell : \tau_{i,\ell}(t) = \tau^{(i)}_k(t)}\delta_{(X_\ell(t)-X_i(t))}.$$ 

Using the many-to-one principle and the change of probability presented in equation (\ref{E:many to one version 1}) we see that
\begin{eqnarray*}
    I(t)
 &=& \E_\Q \Big[ \ee^{X_{\Xi_t}(t)} \,
    {\bf 1}_{ \{ \Xi_t=1\} } \,
    F(X_{\Xi_s}(s), \, s\in [0, \, t]) \,
    \Lambda_{\Xi_t}(t)\Big]
    \\
 &=& \E_\Q \Big[ \ee^{X_{\Xi_t}(t)} \, F(X_{\Xi_s}(s), \, s\in [0, \, t]) \,
    \Lambda_{\Xi_t}(t) \,
    \prod_{k } {\bf 1}_{ \{  \min \mathscr{N}^{(\Xi_t)}_k
    > 0 \} } \Big] 
\end{eqnarray*}
where we recall that by convention, for a point measure $\sN$, $\min \sN$ is the infimum of the support of $\sN$. 

Conditioning on the $\sigma$-algebra generated by the spine (including the successive branching times) we obtain
$$
I(t)
=
\E_\Q \Big[ \ee^{X_{\Xi_t}(t)} \, F(X_{\Xi_s}(s), \, s\in [0, \, t]) \,
    \prod_i \overline{G}_{t-\tau_i^{(\Xi_t)}(t)}^{(f)} (X_{\Xi_t}(t)-X_{\Xi_t,t} (\tau_i^{(\Xi_t)}(t))) \Big] ,
$$

\noindent where, for any $r\ge 0$ and any $x\in \r$,
\begin{equation}
    \overline{G}_r^{(f)}(x)
    :=
    \E \Big[
    \ee^{-f(r) \sum_{j=1}^n \alpha_j
    [\int_{A_j+x} \d \mathscr{N}(r) ] } \,
    {\bf 1}_{ \{ \min \mathscr{N}(r) \ge x\} } \Big].
    \label{Gbarf}
\end{equation}

\noindent Since $(\tau_i^{(\Xi_t)}(t),i\ge 0)$ is a rate $2 $ Poisson process under  $\Q$, we arrive at:\footnote{We recall the Laplace functional of a point Poisson process $\mathscr{P}$: $\E[\exp( -\int f \d \mathscr{P})] = \exp[- \int (1-\ee^{-f}) \d \mu]$, where $\mu$ is the intensity measure.}
\begin{eqnarray}
    I(t)
 &=&\E_\Q \Big[ \ee^{X_{\Xi_t}(t)} \, F(X_{\Xi_s}(s), \, s\in [0, \, t]) \,
    \ee^{- 2  \int_0^t [1-\overline{G}_{t-s}^{(f)}(X_{\Xi_t}(t)-X_{\Xi_s}(s))] \d s} \Big]
    \nonumber
    \\
 &=&\E \Big[ \ee^{\sigma B_t} \, F(\sigma B_s, \, s\in [0, \, t]) \,
    \ee^{- 2  \int_0^t [1-\overline{G}_{t-s}^{(f)}(\sigma B_t-\sigma B_s)] \d s} \Big],
    \label{caracterisation-epinale}
\end{eqnarray}

\noindent where, in the last identity, we used the fact that $(X_{\Xi_s}(s), \, s\in [0, \, t])$ under $\Q$ is a centered Brownian motion (with variance $\sigma^2=2$). This yields Theorem \ref{T: path of X_1}.\hfill$\Box$

\bigskip

\noindent {\bf Remark.} Although we do not need it in the present paper, we mention that (\ref{caracterisation-epinale}) gives the existence and the form of the density of $X_1(t)$ by taking $f\equiv 0$ and $F$ to be the projection on the coordinate $s=t$:
\begin{eqnarray*}
    \P\{ X_1(t) \in \!\d y \}
 &=&\ee^y \,
    \E \Big[ \ee^{- 2
    \int_0^t G_{t-s}(\sigma B_t-\sigma B_s) \d s} \,
    {\bf 1}_{ \{ B_t \in {\!\d y\over \sigma}\} } \Big]
    \\
 &=&\ee^y \,
    \E_{0,{y\over \sigma}}^{(t)} \Big[
    \ee^{- 2  \int_0^t G_{t-s}(\sigma B_t-\sigma B_s) \d s}
    \Big] \,
    \P \{ B_t \in {\!\d y\over \sigma}\}.
\end{eqnarray*}

\section{Properties of the path followed by the leftmost particle: 
proof
of Proposition \ref{T: prop de X_1}}

\label{s:proof-prop2.5}


When applying the many-to-one principle as in \eqref{E:many to one version 1}, if the functional $\Psi_{\Xi} $ only depends on the path of $X_{\Xi_s}(s)$ then the last expectation is simply the expectation of a certain event for the standard Brownian motion. For instance, suppose that we want to check if there exists a path $(X_{i,t}(s), s\in [0,t])$ with some property in the tree. Let $A$ be a measurable subset of continuous functions $[0,t] \mapsto \r$. Then
\begin{equation}\label{ineqmanyto1}
\P(\exists i \le N(t) : (X_{i,t}(s), s\in [0,t]) \in A )  \le \P(\ee^{\sigma B_t}  ; (\sigma B_s , s\in [0,t]) \in A )
\end{equation}
where $(B_s,s \ge 0)$ is a standard Brownian motion under $\P.$
This is the main tool
we use in proving Proposition \ref{T: prop de X_1}.

Let $(B_s,s \ge 0)$ denote a standard Brownian motion. Before proceeding to the proof of Proposition \ref{T: prop de X_1}, let us recall (see, for example, Revuz and Yor~\cite{revuz-yor}, Exercise III.3.14) the joint distribution of $\min_{[0,t]}B_s$ and $B_t$: for any $x>0$, $y>0$ and $t>0$,
\begin{eqnarray}
    \P\Big( \min_{[0,t]} B_s>-x,\;  B_t+x \in \!\d y\Big)
 &=&\Big( {2\over \pi  t}\Big)^{\! 1/2} \, \ee^{-{ x^2+y^2\over 2 t}} \sinh \Big({xy\over  t}\Big) \d y
    \nonumber
    \\
 &\le& \Big( {2\over \pi  t^3}\Big)^{\! 1/2} xy \d y \, ,
    \label{joint-min-Bt}
\end{eqnarray}

\noindent the last inequality following from the facts that $\sinh z\le ze^z$ for $z\ge 0$, and that $\ee^{-{ x^2+y^2\over 2 t}+{y x\over  t}}\le 1$.

We now turn to the proof of Proposition \ref{T: prop de X_1}. Let $J_\eta(t): = \{ i \le N(t) : |X_i(t) -m_t|<\eta \}$ where $m_t={3\over 2} \log t +C_B$ by (\ref{bramson-constant}). We recall  that for $t \ge 1$ and $x>0$, we define the good event $A_{t}(x,\eta)$ by
\begin{equation*}
    A_{t}(x,\eta)
    :=
    E_1(x, \eta )\cap E_2(x, \eta) \cap E_3(x,\eta)
\end{equation*}

\noindent where the events $E_i$ are defined by
\begin{eqnarray*}
E_1(x,\eta) &:=& \Big\{ \forall i \in J_\eta(t) , \ \min_{[0,t]} X_{i,t}(s) \ge -x,\, \min_{[{t\over 2},\, t]} X_{i,t}(s) \ge m_t -x \Big\},\\
E_2(x,\eta) &:=& \Big\{ \forall i \in J_\eta(t) , \forall s\in[x,\, {t\over 2}], X_{i,t}(s) \ge s^{1/3} \Big\},\\
E_3(x,\eta) &:=& \Big\{ \forall i \in J_\eta(t)  , \forall s\in [{t\over 2},\, t-x],\, X_{i,t}(s)-X_i(t) \in [(t-s)^{1/3},(t-s)^{2/3}]\Big\}.
\end{eqnarray*}

\noindent We now prove the claim of Proposition \ref{T: prop de X_1}:
For any $\varepsilon>0$ and $\eta>0$, there exists $x>0$ large enough such that $\P(A_t(x,\eta)) \ge 1-\varepsilon$ for $t$ large enough.

\bigskip

\noindent {\it Proof}. The  notation $c$ denotes a constant (that may depend on $\eta$) which can change from line to line. We deal separately with the events $E_i(x,\eta)$. We want to show that for any $i\in\{1,2,3\}$, there exists $x$ large enough such that $\P((E_i(x,\eta))^\complement) \le \varepsilon$ for $t$ large enough.\\

{\it Bound of $\P(E_1(x,\eta)^\complement)$}.\\
First, observe that $\min\{X_i(t), i\le N(t), t\ge 0\}$ is an a.s.\ finite random variable and therefore
$$\P\Big(\min_{i \in J_\eta(t), s \in [0,t]} X_{i,t}(s) \le -x\Big) \le  \P\Big(\min\{X_i(t), i\le N(t), t\ge 0\} \le -x\Big) \le \varepsilon$$ for $x$ large enough. It remains to bound the probability to touch level $m_t -x$ between ${t\over 2}$ and $t$. By the previous remarks, we can assume that $\min_{[0,t]} X_{i,t}(s)\ge -z$ for all $i \in J_\eta(t)$. We claim that, for any $z,\eta\ge 0$ , there exists $c>0$ and a function $\varepsilon_t \to 0$  such that for any $x\ge 0$ and $t\ge 1$,
\begin{equation}\label{claimE1}
\P \Big\{\exists i \in J_\eta(t) , \ \min_{s\in[0,\, t]} X_{i,t}(s)\ge -z, \min_{s\in [{t\over 2},\, t]} X_{i,t}(s) = m_t -x \pm 1   \Big\} \le c\ee^{-cx} + \varepsilon_t
\end{equation}

\noindent where $y=u\pm v$ stands for $y\in[u-v,u+v]$. This will imply the bound on $E_1(x,\eta)^\complement$. Let us prove the claim. We see that the probability on the left-hand side is $0$ if $x> m_t  +z +1$ (indeed, if $x> m_t  +z +1$ and $\min_{s\in [{t\over 2},\, t]} X_{i,t}(s) \le  m_t -x +1 < -z$, then it is impossible to have $\min_{s\in[0,\, t]} X_{i,t}(s)\ge -z$). We then take $x\le {7\over 4}\log t$ (any constant lying in $({3\over 2}, \, 2)$ would do the job in place of ${7\over 4}$). 

Let $a\in (0, \, {t\over 2})$ (at the end, $a=\ee^{x/2}$). 
We discuss whether $\{\min_{s\in[t/2,t-a]} X_{i,t}(s) = m_t -x \pm 1\}$ or $\{\min_{s\in[t-a,t]} X_{i,t}(s) =  m_t -x \pm 1\}$. We denote by $p_{\rm{claim}}^{[t/2,t-a]}(x)$ (resp. $p_{\rm{claim}}^{[t-a,t]}(x)$) the probability in (\ref{claimE1}) on the event $\{\min_{s\in[t/2,t-a]} X_{i,t}(s) =  m_t -x \pm 1\}$ (resp. $\{\min_{s\in[t-a,t]} X_{i,t}(s) =  m_t -x \pm 1\}$). Equation (\ref{ineqmanyto1}) provides us with the following bound
\begin{equation}\label{eq:claim=P(B)}
 p_{\rm{claim}}^{[t/2,t-a]}(x) \le \ee^{\eta+C_B} t^{3/2}\P(B)
\end{equation}
where
$$
\P(B) := \P\Big\{ \sigma \underline B^{[0,t]}\ge -z, \sigma \underline B^{[{t\over 2},t-a]} = m_t -x \pm 1, \sigma \underline B^{[t/2,t]} = m_t -x \pm 1, \sigma B_t= m_t \pm \eta \Big\}
$$
and $
\underline B^{[b_1,b_2]} := \min_{s\in[b_1,b_2]} B_s.
$ By reversing time, we see that
\begin{eqnarray*}
    \P(B)
&\le&\P\Big\{ \sigma \underline B^{[0,t]} \ge -m_t -(z+\eta) , \sigma \underline B^{[0,a]} \ge -\eta-x-1, \\
&& ~~~~~~~~~~~~~~~~ \sigma B_t= -m_t \pm \eta,\sigma \underline B^{[a,t/2]} =-x \pm(\eta+1)  \Big\}.
\end{eqnarray*}

\noindent  By the Markov property at time $t/2$, we obtain
\begin{eqnarray*}
\P(B)
&\le&
\E\Big[ {\bf 1}_{\{\sigma \underline B^{[a,t/2]} = -x \pm(\eta+1)\}}{\bf 1}_{\{\sigma \underline B^{[0,a]} \ge -\eta-x-1\}}\\
 &&\qquad \times \P_{B_{t/2}} \Big\{ \sigma \underline B^{[0,t/2]}\ge -m_t -(z+\eta) ,\sigma B_{t/2} = -m_t \pm \eta \Big\}  \Big],
\end{eqnarray*}

\noindent where, for any $y\in \r$, $\P_y$ is the probability under which $B$ starts at $y$: $\P_y(B_0=y) =1$. (So $\P_0= \P$). By (\ref{eq:claim=P(B)}) and (\ref{joint-min-Bt}), it follows that
\begin{eqnarray*}
 p_{\rm{claim}}^{[t/2,t-a]}(x)
 &\le&
 c  (z+2\eta)^2 \E\Big[ {\bf 1}_{\{\sigma \underline B^{[a,t/2]} = -x \pm(\eta+1),\sigma \underline B^{[0,a]} \ge -\eta-x-1\}}(\sigma B_{t/2} +m_t +(z+\eta))\Big]\\
 &\le&
c  (z+2\eta)^2 (\E_1 +\E_2)
\end{eqnarray*}

\noindent where
\begin{eqnarray*}
\E_1 &:=& \E\Big[ {\bf 1}_{\{\sigma \underline B^{[a,t/2]} = -x \pm(\eta+1),\sigma \underline B^{[0,a]} \ge -\eta-x-1\}} (\sigma B_{t/2} + \eta+x+1)\Big],\\
\E_2 &:=& (m_t +z-x-1)\P\Big\{ \sigma \underline B^{[a,t/2]} = -x \pm(\eta+1),\sigma \underline B^{[0,a]} \ge -\eta-x-1 \Big\}.
\end{eqnarray*}

\noindent To bound $\E_2$ is easy. We have $|\E_2| \le O(\log t)\, \P(\sigma  \underline B^{[0,t/2]} \ge -\eta - x - 1)=O((\log t)^2)\, t^{-1/2}$ uniformly in $x\le {7\over 4} \log t$.
Now consider $\E_1$.
We note that
$(\sigma B_{t/2} +\eta+x+1){\bf 1}_{\{ \sigma \underline B^{[0,t/2]} \ge -\eta-x-1\}}$
is the $h$-transform of the three-dimensional Bessel process, and we denote by $(R_s,s\ge 0)$ a three-dimensional Bessel process. Then,
\begin{eqnarray*}
\E_1 = (\eta+x+1)\P_{\eta+x+1}(\sigma \underline R^{[a,t/2]}) \le 2(\eta + 1)) \le (\eta+x+1)\P_{\eta+x+1}( \min_{s\ge a} \sigma R_s \le 2(\eta+1) )
\end{eqnarray*}

\noindent with natural notation. The infimum of a three-dimensional Bessel process starting from $x$ is uniformly distributed in $[0,x]$ (see Revuz and Yor~\cite{revuz-yor}, Exercise V.2.14). Applying the Markov property at time $a$, we get $\E_1 \le {2(\eta+1)\over \sigma}(\eta+x+1)\E_{\eta+x+1}[1/R_a]\le c (\eta+x+1) a^{-1/2}$. We take $a=\ee^{x/2}$. The preceding inequality implies that for any $x\ge 0$ and $t\ge 1$,

\begin{equation}\label{eq:claimE1a}
 p_{\rm{claim}}^{[t/2,t-a]}(x) \le c(\eta+x+1)^2\ee^{-x/4} + c (\log t)^2 t^{-1/2}.
\end{equation}

\noindent We deal now with the probability $p_{\rm{claim}}^{[t-a,t]}(x)$. In this case, the minimum on $[t-a,t]$ belongs to $[m_t -x-1,m_t -x+1]$. Since we know that $p_{\rm{claim}}^{[t/2,t-a]}(x)$ is small, we can restrict to the case where the minimum on $[t/2,t-a]$ is greater than $m_t -x+1$, i.e.,
\begin{eqnarray*}
&& p_{\rm{claim}}^{[t-a,t]}(x)
\le
\sum_{y=x}^{\lfloor 2\log t \rfloor} p_{\rm{claim}}^{[t/2,t-a]}(y) + \\
&&  \P\Big\{ \exists i \in J_\eta(t): \,   \underline X_{i,t}^{[0,t]} \ge -z,\, \underline X_{i,t}^{[t/2,t-a]} >  m_t -x +1, \,
  \underline X_{i,t}^{[t-a,t]} \le m_t -x +1 \Big\}.
\end{eqnarray*}

\noindent From (\ref{eq:claimE1a}), we know that $\sum_{y=x}^{\lfloor 2\log t \rfloor} p_{\rm{claim}}^{[t/2,t-a]}(y)\le o(t) + c\, \ee^{-x/8}$ with as usual $\underline X_{i,t}^{[a,b]}:=\min_{s\in[a,b]} X_{i,t}(s)$.

Suppose that we kill particles as soon as they hit the position $-z$ during the time interval $[0,t/2]$ and as soon as they are left of or at position $m_t-x+1$ during the time interval $[t/2,t].$ Call $\mathcal S^{[t-a,t]}$ the number of particles that are killed during the time interval $[t-a,t].$ 
Hence,
\begin{eqnarray}\label{eq:claimE2b}
p_{\rm{claim}}^{[t-a,t]}(x)
\le
o(t)+ c\ee^{-x/8} + \E[\#\mathcal S^{[t-a,t]}].
\end{eqnarray}


We observe that by stopping particles either at time $t$ or when they first hit $-z$ during $[0,t/2]$ or  $m_t-x+1$ during the time interval $[t/2,t],$ we are defining a so-called {\it dissecting stopping line}. Stopping lines were introduced and studied --- among others --- by \cite{chauvin91, jagers89} essentially for branching random walks. More recently, they have been used with great efficacy by e.g. Kyprianou in the context of branching Brownian motion to study traveling wave solutions to the F-KPP equation \cite{kyprianou}. 
More precisely, for a continuous path $X : \r_+ \to \r$ let us call $T(X)$ the stopping time
$$
T(X) : = \inf\{s \le t/2 : X(s) \le -z\} \wedge \inf\{s \in [t/2,t] : X(s) \le m_t-x+1 \} \wedge t
$$
and for $i \le N(t)$ define $T_i := T(X_{i,t}(\cdot)).$ We also need a notation for the label of the progenitor at time $T_i$ of the particle at $X_i(t)$ at time $t$: let $J_i \le N(T_i)$ be the almost surely unique integer such that 
$$
X_{J_i}(T_i) = X_{i,t}(T_i).
$$ 
We now formally define the stopping line $\ell$ by
$$
\ell := \text{enum}((J_i,T_i)_{i \le N(t)}) 
$$
where $\text{enum}$ means that $\ell$ is an enumeration without repetition. In general, stopping lines can be far more sophisticated objects, and $\ell$ is a particularly simple example of this class, which is bounded by $t$ (and thus {\it dissecting}).

We now need a generalization of the many-to-one principle \eqref{E:many to one version 1} to stopping lines. Although this can now be considered {\it common knowledge}, surprisingly only \cite[Lemma 3.1 and 3.2]{maillard12} gives the result in sufficient generality for our purposes. 
\begin{fact}\label{many21 stop line}
Let  $g : (x,t) \mapsto g(x,t), \r \times \r_+ \to \r$  be measurable. Then, if $X(t) = \sigma B_t +\varrho t$ where $B$ is a standard Brownian motion
$$
\E_\P [\sum_{(i,t) \in \ell} g(X_{i}(t),t) ] = \E [e^{\lambda T(X)} g(X_{T(X)},T(X)].
$$
\end{fact}
To see this, one can for instance adapt the proofs for the fixed-time many-to-one lemma in \cite{hardyharris,harrisroberts} to the case of dissecting stopping-lines.

Once one factors in the Girsanov term to get rid of the drift, one sees that
\begin{align*}
\E_\P [\sum_{(i,t) \in \ell} g(X_{i}(t),t) ] &= \E [e^{\sigma B_{T(\sigma B)} }  g(\sigma B_{T(\sigma B)},T(\sigma B)] \\
&= \E_\Q [e^{ X_{\Xi_{T(\Xi)}}(T(X_\Xi))}   g(X_{\Xi_{T(\Xi)}}(T(X_\Xi)),T(X_\Xi))].
\end{align*}
By applying this with $g(x,s)= \mathbf{1}_{s \in (t-a,t)}$ we see that 
$$
\E[\# \mathcal S^{[t-a,t]}] =\ee^{C_B}t^{3/2}\ee^{-x+1} \P\Big\{ \sigma \underline B^{[0,t/2]} \ge -z,\, T^{t/2}_{(m_t -x+1)/\sigma} \in [t-a,t],\sigma B_{t/2}\ge m_t -x+1\Big\}.
$$

\noindent where $T^{t/2}_{y}:= \min \{ s\ge t/2\,:  B_s = y \}$. As usual, we apply the Markov property at time $t/2$ so that
\begin{eqnarray*}
&&\P\Big\{\sigma \underline B^{[0,t/2]} \ge -z,\, T^{t/2}_{(m_t -x+1)/\sigma} \in [t-a,t], \sigma B_{t/2}\ge m_t -x+1\Big\}\\
&=&
\E\left[{\bf 1}_{\{ \sigma \underline B^{[0,t/2]} \ge -z \}} \P_{B_{t/2}}\Big\{T_{(m_t -x+1)/\sigma} \in [t/2-a,t/2]\Big\}\right]
\end{eqnarray*}

\noindent where $T_y := \min\{s\ge 0:  B_s=y\}$ is the hitting time at level $y$. We know that $\P(T_y\in du)={y\over \sqrt{2\pi}}u^{-3/2}\ee^{-{y^2\over 2u}}\d u\le c y u^{-3/2}\d u$ for $u\ge 0$. It follows that for
some constant $c>0$ and
any $a\in [1,t/3]$
\begin{eqnarray*}
&&\P\Big\{\sigma \underline B^{[0,t/2]} \ge -z,\, T^{t/2}_{(m_t -x+1)/\sigma} \in [t-a,t], \sigma B_{t/2}\ge m_t -x+1\Big\}\\
&\le&
c  a t^{-3/2}\E[ B_{t/2} {\bf 1}_{\{\sigma \underline B^{[0,t/2]} \ge -z\}}]=
c a t^{-3/2} z.
\end{eqnarray*}

\noindent Thus, $\E[\#\mathcal S^{[t-a,t]}] \le c a z \ee^{-x}  = cz \ee^{-x/2} $ for $a=\ee^{x/2}$.  Claim (\ref{claimE1}) now follows from equations (\ref{eq:claimE1a}) and (\ref{eq:claimE2b}).

\bigskip

{\it Bound of $\P((E_2(x,\eta))^\complement)$}.\\
We can restrict to the event $E_1(z,\eta)$  for $z$ large enough. By the many-to-one principle, we get
$$
 \P(E_2(x,\eta)^\complement,E_1(z,\eta))
\le \ee^{\eta+C_B} t^{3/2} \P(\widehat{B})
$$
where $\P(\widehat{B})$ is defined by
$$
\P(\widehat{B}):=\P\Big\{ \exists s\in [x,t/2]: \sigma B_s \le s^{1/3},\sigma  \underline B^{[0,t/2]} \ge -z, \sigma \underline B^{[t/2,t]} \ge m_t -z, \sigma B_t \le  m_t +\eta \Big\}.
$$

\noindent We will actually bound the probability 
\begin{eqnarray}
 &&\P(\widehat{B},\d r)
    \label{Bhat}
    \\
 &:=&\P\Big\{ \exists s\in [x,t/2]: \sigma B_s \le s^{1/3}, \sigma \underline B^{[0,t/2]} \ge -z, \sigma \underline B^{[t/2,t]} \ge m_t -z, \sigma B_t \in m_t+\d r \Big\}.
    \nonumber
\end{eqnarray}
 
\noindent Applying the Markov property at time $t/2$ yields that
\begin{eqnarray*}
&& \P(\widehat{B},\d r)\\
&=&
\E\left[ {\bf 1}_{\{\exists s\in [x,t/2]: \sigma B_s \le s^{1/3}\}} {\bf 1}_{\{ \sigma \underline B^{[0,t/2]} \ge -z\}} \P_{B_{t/2}}\Big\{\sigma \underline B^{[0,t/2]} \ge m_t -z, \sigma B_{t/2} \in m_t+\d r \Big\} \right]\\
&\le&
c (r+z) t^{-3/2} \E\left[ {\bf 1}_{\{\exists s\in [x,t/2]: \sigma B_s \le s^{1/3}\}} {\bf 1}_{\{\sigma  \underline B^{[0,t/2]} \ge -z\}} (\sigma B_{t/2} -  m_t +z)_+ \right]\d r\\
&\le&
c (r+z)t^{-3/2} \E\left[ {\bf 1}_{\{\exists s\in [x,t/2]: \sigma B_s \le s^{1/3}\}} {\bf 1}_{\{ \sigma \underline B^{[0,t/2]} \ge -z\}} (\sigma B_{t/2} +z) \right]\d r
\end{eqnarray*}

\noindent where the second inequality comes from equation (\ref{joint-min-Bt}), and we set $y_+:=\max(y,0)$. We recognize the $h$-transform of the Bessel process. Therefore
\begin{equation}\label{P(B,dr)}
\P(\widehat{B},\d r) \le c z (r+z)t^{-3/2}\P_z( \exists s\in [x,t/2]: \sigma R_s \le z+s^{1/3})\d r
\end{equation}
where as before $(R_s,s\ge 0)$ is a three-dimensional Bessel process. In particular, $\P(\widehat{B}) = \int_{-z}^\eta \P(\widehat{B},\d r)\le c z(z+\eta)^2t^{-3/2}\P_z( \exists s\in [x,t/2]: \sigma R_s \le z+s^{1/3})$. This yields that 
\begin{eqnarray*}
\P(E_2(x,\eta)^\complement,E_1(z)) &\le& \ee^{\eta+C_B} c z(z+\eta)^2 \P_z( \exists s\in [x,t/2]: \sigma R_s \le z+s^{1/3})\\
&\le&
\ee^{\eta+C_B} c z (z+\eta)^2 \P_z( \exists s\ge x: \sigma R_s \le z+s^{1/3})
\end{eqnarray*}

\noindent  and we deduce that $\P(E_2(x,\eta)^\complement,E_1(z))\le \varepsilon$ for $x$ large enough.

\bigskip

{\it Bound of $\P((E_3(x,\eta))^\complement)$}.\\
The bound on $\P((E_3(x,\eta)^\complement))$ works similarly. We have by the many-to-one principle
\begin{equation}\label{eq:P(B)=E3}
 \P(E_3(x,\eta)^\complement,E_1(z,\eta),E_2(z,\eta)) \le \ee^{\eta+C_B} t^{3/2} \P(\widetilde{B})
 \end{equation}
with $\P(\widetilde{B})$ defined by
\begin{eqnarray*}
&&\P\Big\{ \exists s\in [t/2,t-x]: \sigma(B_s -B_t) \notin [(t-s)^{1/3},(t-s)^{2/3}],
\sigma \underline B^{[0,t/2]} \ge -z,\\
 && ~~~~~\sigma \underline B^{[t/2,t]} \ge m_t -z, \sigma B_t =  m_t \pm \eta\Big\}.
\end{eqnarray*}
Let 
\begin{eqnarray}\label{Btilde}
\nonumber &&\P(\widetilde{B},\d r) := \P\Big\{ \exists s\in [t/2,t-x]: \sigma(B_s -B_t) \notin [(t-s)^{1/3},(t-s)^{2/3}],
\sigma \underline B^{[0,t/2]} \ge -z, \\
 && ~~~ \qquad \qquad \qquad   \sigma \underline B^{[t/2,t]} \ge m_t -z, \sigma B_t \in  m_t + \d r\Big\}.
\end{eqnarray} 

\noindent Reversing time, we get
\begin{eqnarray}
\nonumber   \P(\widetilde{B},\d r)  &\le &\P\Big\{ \exists s\in [x,t/2]: \sigma B_s \notin [s^{1/3},s^{2/3}], \sigma \underline B^{[t/2,t]} \ge -m_t-z-\eta, \\
   && ~~~~ \sigma \underline B^{[0,t/2]} \ge  -z-\eta , \sigma B_t  +m_t \in \d r\Big\}. \label{eq:P(B)E3}
\end{eqnarray}

\noindent  By equation (\ref{joint-min-Bt}), we have for any $y> -\frac32 \log t -z-\eta$, and $t\ge 1$
$$
\P_{y} \Big\{ \sigma \underline B^{[0,t/2]} \ge -m_t -z-\eta,\sigma B_{t/2}+m_t \in \d r \Big\}
\le c (y+m_t +z+\eta)(r+z+\eta) t^{-3/2}\d r.
$$
Applying the Markov property at time $t/2$ in (\ref{eq:P(B)E3}), we get for $t\ge 1$
\begin{eqnarray*}\nonumber
&& \P(\widetilde{B},\d r) \\ 
&\le&
 c (r+z+\eta) t^{-3/2}\E\Big[ {\bf 1}_{\{\exists s\in [x,t/2]: \sigma B_s \notin [s^{1/3},s^{2/3}], \sigma \underline B^{[0,t/2]} \ge  -z-\eta\}} (\sigma B_{t/2} +m_t+z+\eta)\Big]\d r\\
&\le& 
 c (r+z+\eta) t^{-3/2}\left({m_t\over \sqrt{t}} + \E\Big[ {\bf 1}_{\{\exists s\in [x,t/2]: \sigma B_s \notin [s^{1/3},s^{2/3}], \sigma \underline B^{[0,t/2]} \ge  -z-\eta\}} (\sigma B_{t/2} +z+\eta)\Big] \right)\d r.
 \end{eqnarray*}
 
 \noindent On the other hand,
 \begin{eqnarray*}
&& \E\Big[ {\bf 1}_{\{\exists s\in [x,t/2]: \sigma B_s \notin [s^{1/3},s^{2/3}], \sigma \underline B^{[0,t/2]} \ge  -z-\eta\}} (\sigma B_{t/2} +z+\eta)\Big] \\
 &=&
 (z+\eta) \P_{z+\eta}(\exists s\ge x: \sigma R_s-z-\eta \notin [s^{1/3},s^{2/3}])
 \end{eqnarray*}
 
 \noindent where, as before, $(R_s,\, s\ge 0)$ is a three-dimensional Bessel process.This implies that
 \begin{eqnarray} \label{P(B,dr)2}
 &&  \P(\widetilde{B},\d r) \\ 
 &\le&  c (r+z+\eta) t^{-3/2}\left({m_t\over \sqrt{t}} + (z+\eta) \P_{z+\eta}(\exists s\ge x: \sigma R_s-z-\eta \notin [s^{1/3},s^{2/3}]) \right)\d r. \nonumber
\end{eqnarray}

\noindent  We get that
$$
\P(\widetilde{B}) \le  c (z+2\eta)^2 t^{-3/2}\left({m_t \over \sqrt{t}} +(z+\eta)\P_{z+\eta}(\exists s\ge x: \sigma R_s-z-\eta \notin [s^{1/3},s^{2/3}]) \right)
$$
which is less than $c (z+2\eta)^2 t^{-3/2}\varepsilon $ for $x$ large enough (as $t\to\infty$) and we conclude by (\ref{eq:P(B)=E3}). $\Box$

\bigskip
 
For future reference we now prove the following lemma
which shows that the probability for a Brownian path conditioned to end up near $m_t$ of satisfying event $E_1$ but not $E_2$ or $E_3$ decreases like $1/t.$ Let $\P_{a,b}^{(t)}$ denote the probability under which $B$ is a Brownian bridge from $a$ to $b$ of length $t$. The notation $o_x(1)$ designates an expression depending on $x$ (and also on $r$ and $z$, but independent of $t$) which converges to 0 as $x \to \infty$. We recall that $\underline{B}^{[a,b ]}:=\min_{s\in[a,b]} B_s$.
\begin{lemma}\label{l:small o of 1/t}
Fix $r\in \r$ and $z>0$.  We have
\begin{align*}
\P_{0,{m_t+r \over \sigma}}^{(t)}\Big( &\sigma \underline{B}^{[0,t/2 ]}\ge -z, \, \sigma \underline{B}^{[t/2,t ]}\ge m_t -z, \\ & \qquad \exists s\in[x,t-x]\,: \, \sigma B_s< \min(s^{1/3}, m_t +(t-s)^{1/3}) \Big) = {1\over t} o_x(1)
\end{align*}
in the sense that $\limsup_{t\to\infty} t\P_{0,{m_t+r \over \sigma}}^{(t)}(\ldots)=o_x(1)$.
Furthermore, there exists a constant $c>0$ such that for any $t\ge 1$, $z>0$ and $r\in \r$ such that $|r|\le \sqrt{t}$,
$$
\P_{0,{m_t+r\over \sigma}}^{(t)} \Big( \sigma \underline{B}^{[0,t/2 ]}\ge -z, \, \sigma \underline{B}^{[t/2,t ]}\ge m_t -z \Big)\le {c\over t} z|r+z|.
$$

\end{lemma}

\begin{proof} We have
\begin{align*}
& \P_{0,{m_t +r \over \sigma}}^{(t)}\Big( \sigma \underline{B}^{[0,t/2 ]}\ge -z, \, \sigma \underline{B}^{[t/2,t ]}\ge m_t-z, \\ & \qquad \qquad \exists s\in[x,t-x]\,: \, \sigma B_s< \min(s^{1/3}, m_t +(t-s)^{1/3}) \Big) \\ & \le
\P_{0,{m_t+r \over \sigma}}^{(t)}\Big( \sigma \underline{B}^{[0,t/2 ]}\ge -z, \, \sigma \underline{B}^{[t/2,t ]}\ge m_t -z, \, \exists s\in[x,t/2]\,: \, \sigma B_s< s^{1/3} \Big) \\
&\qquad + \P_{0,{m_t+r \over \sigma}}^{(t)}\Big( \sigma \underline{B}^{[0,t/2 ]}\ge -z, \,  \sigma \underline{B}^{[t/2,t ]}\ge m_t-z, \\ & \qquad \qquad\qquad \exists s\in[t/2, t-x]\,: \, \sigma B_s< m_t +(t-s)^{1/3} \Big) .
\end{align*}
We treat the two terms
on the right-hand side
successively. Using the definition of the Brownian bridge, we observe that, as $t\to\infty$
\begin{align*}
&\P_{0,{m_t+r \over \sigma}}^{(t)}\Big(\sigma \underline{B}^{[0,t/2 ]}\ge -z,\,\sigma \underline{B}^{[t/2,t ]}\ge m_t-z,\, \exists s\in[x,t/2]\,: \, \sigma B_s< s^{1/3} \Big) \\
&= \sigma \sqrt{2\pi t}\, \ee^{ {(m_t +r)^2 \over 2\sigma^2 t} } \lim_{\d r \to 0}  {1\over \d r} \P(\widehat{B},\d r) 
\end{align*}

\noindent with $\P(\widehat{B},\d r)$ defined in (\ref{Bhat}). By equation (\ref{P(B,dr)}), $\P(\widehat{B},\d r) \le ct^{-3/2}(r+z)\P_r( \exists s\in [x,t/2]: \sigma R_s \le s^{1/3})\d r$, where $(R_s,s\ge 0)$ is a three dimensional Bessel process. Hence 
$$
\P_{0,{m_t+r \over \sigma}}^{(t)}\Big(\sigma \underline{B}^{[0,t/2 ]}\ge -z,\,\sigma \underline{B}^{[t/2,t ]}\ge m_t-z,\,  \exists s\in[x,t/2]\,: \, \sigma B_s< s^{1/3} \Big) \sim {1\over t}o_x(1).
$$
Similarly, notice that
\begin{eqnarray*}
&& \P_{0,{m_t+r \over \sigma}}^{(t)}\Big(\sigma \underline{B}^{[0,t/2 ]}\ge -z,\, \sigma \underline{B}^{[t/2,t ]}\ge m_t-z,\,\exists s\in[t/2, t-x]\,: \, \sigma B_s< m_t +(t-s)^{1/3}\Big) \\
&\le&
\sigma \sqrt{2\pi t}\, \ee^{ {(m_t +r)^2 \over 2\sigma^2 t} } \lim_{\d r \to 0}  {1\over \d r}  \P(\widetilde{B},\d r)
\end{eqnarray*}
with $ \P(\widetilde{B},\d r)$ defined in (\ref{Btilde}). Then, equation (\ref{P(B,dr)2}) implies that
$$
\P_{0,{m_t+r \over \sigma}}^{(t)}\Big(\sigma \underline{B}^{[0,t/2 ]}\ge -z,\, \sigma \underline{B}^{[t/2,t ]}\ge m_t-z,\,\exists s\in[t/2, t-x]\,: \, \sigma B_s< m_t +(t-s)^{1/3}\Big) 
$$
is ${1\over t}o_x(1)$, which proves the first assertion. Let us prove the second one. We can suppose that $r+z\ge 0$, since the statement is trivial otherwise. We have that
\begin{eqnarray*}
&& \P_{0,{m_t+r\over \sigma}}^{(t)} \Big( \sigma \underline{B}^{[0,t/2 ]}\ge -z, \, \sigma \underline{B}^{[t/2,t ]}\ge m_t -z \Big)\\
&=&
\sigma \sqrt{2\pi t}\, \ee^{ {(m_t +r)^2 \over 2\sigma^2 t} } \lim_{\d r \to 0}  {1\over \d r} \P( \sigma \underline{B}^{[0,t/2 ]}\ge -z, \, \sigma \underline{B}^{[t/2,t ]}\ge m_t -z,\, \sigma B_t \in m_t+\d r  ) .
\end{eqnarray*}

\noindent By the Markov property at time $t/2$ and equation (\ref{joint-min-Bt}), we see that
\begin{eqnarray*}
&& \P( \sigma \underline{B}^{[0,t/2 ]}\ge -z, \, \sigma \underline{B}^{[t/2,t ]}\ge m_t -z,\, \sigma B_t \in m_t+\d r  )\\
&\le&
c t^{-3/2} (r+z) \E\left[ {\bf 1}_{\{ \sigma \underline{B}^{[0,t/2 ]}\ge -z  \}} (\sigma B_{t/2} +z-m_t)_+\right]\d r
\end{eqnarray*}

\noindent where $y_+$ stands for $\max(y,0)$. We notice as before that $ \E\left[ {\bf 1}_{\{ \sigma \underline{B}^{[0,t/2 ]}\ge -z  \}} (\sigma B_{t/2} +z-m_t)_+\right] \le  \E\left[ {\bf 1}_{\{ \sigma \underline{B}^{[0,t/2 ]}\ge -z  \}} (\sigma B_{t/2} +z)\right]=z$, which completes the proof. 
\end{proof}


\section{The decoration point measure $\mathscr{Q}$: Proof of Theorem \ref{T: structure de Q}}
\label{S: preuve de 2.3}

This section is devoted to the study of the decoration point measure $\mathscr{Q}$. 

\noindent {\it Proof of Theorem \ref{T: structure de Q}.} Recall that $X_{i,t}(s)$ is the position at time $s \in [0,t]$ of the ancestor of $X_i(t)$ and that we have defined 
$$
Y_t(s) :=X_{1,t}(s)-X_1(t).
$$ 
Let $\zeta>0$, and let $f := {\bf 1}_{[0, \, \zeta]}$. Let $t\ge \zeta$. Let
\begin{eqnarray*}
    L_j(t, \, \zeta)
 &:=&\int_{A_j} \d \mathscr{Q}(t, \, \zeta)     \\
 &=&\sum_{i: \, t-\tau_i(t) \le \zeta} \int_{  A_j } \d \mathscr{N}_i(t) ,
\qquad 1\le j\le n.
\end{eqnarray*}

\noindent Let $F_1: \, C(\r_+, \, \r) \to \r_+$  be a bounded continuous function and $F_2:={\bf 1}_{[\eta_1,\eta_2]}$ for some $\eta_2>\eta_1$. Fix $x>0$ and let
\begin{equation}\label{def:as}
a_s:=
\begin{cases}
-x & {\rm if} \, s\in[0,t/2],\\
m_t - x & {\rm if} \, s\in(t/2,t].
\end{cases}
\end{equation}

\noindent We define for any function $X:[0,t]\to\r$, the event
\begin{eqnarray*}
 A(X)
&:=&\{X(s)\ge a_s\, \forall \, s\in[0,t-\zeta]\}\cap\{X(s)-X(t) \ge -x,\,\forall \, s\in [t-\zeta,t]\}\cap\\
&&\hskip-15pt \cap\{X(t-\zeta)-X(t)\in (\zeta^{1/3},\zeta^{2/3})\} \cap \{ \inf\{ s: X(t-s) = \min_{u \in [0,t/2]}X(t-u) \} \le x \}.
\end{eqnarray*}
We easily check that Proposition \ref{T: prop de X_1} implies that $\{X_{1}(t)-m_t\in [\eta_1,\eta_2]\}\cap (A(X_{1,t}))^\complement$ is of probability arbitrary close to $0$ when $x$ and $\zeta$ are large enough. Therefore, we fix $x$ large and we work on the event $A(X_{1,t})$ and we will let $t\to\infty$ then $\zeta\to\infty$ then $x\to\infty$. By (\ref{caracterisation-epinale}), for $t\ge \zeta$:
\begin{eqnarray*}
&& \E\Big\{ {\bf 1}_{A(X_{1,t})} \, F_1(Y_t(s), \, s\in [0, \, \zeta]) \,
    \ee^{-\sum_{j=1}^n \alpha_j L_j(t, \, \zeta)} \,
    F_2(X_1(t)- m_t)     \Big\}
    \\
 &=&\E \Big[ {\bf 1}_{A(\sigma B)} \,
    F_1(\sigma B_t - \sigma B_{t-s}, \, s\in [0, \, \zeta]) \,
    \ee^{\sigma B_t} \,
    \ee^{- 2  \int_0^t [1-\overline{G}_{t-u}^{(f)}(\sigma B_t-\sigma B_u)] \d u}
    F_2(\sigma B_t - m_t) \Big]\\
 &=&\E \Big[ {\bf 1}_{A(\sigma \overline B)} \,
    F_1(\sigma B_s, \, s\in [0, \, \zeta]) \,
    \ee^{\sigma B_t} \,
    \ee^{- 2  \int_0^t [1-\overline{G}_v^{(f)}(\sigma B_v)] \d v}
    F_2(\sigma B_t - m_t) \Big]
    \\
 &=& \int_\r \P \{ B_t \in {\! \d y\over \sigma} \} \, \ee^y \,
    F_2(y- m_t)\,
    \E_{0, {y\over \sigma}}^{(t)} \Big[ {\bf 1}_{A(\sigma \overline B)} \,
    F_1(\sigma B_s, \, s\in [0, \, \zeta]) \,
    \ee^{- 2  \int_0^t [1-\overline{G}_v^{(f)}(\sigma B_v)] \d v} \Big] ,
\end{eqnarray*}

\noindent where $\overline B_s := B_t-B_{t-s}$, $s\in[0,t]$ (so $\overline{B}_t = B_t$), and $\E_{0, {y\over \sigma}}^{(t)}$ denotes expectation with respect to the probability $\P_{0, {y\over \sigma}}^{(t)}$, under which $(B_v, \, v\in [0, \, t])$ is a Brownian bridge of length $t$, starting at $0$ and ending at ${y\over \sigma}$. Since $f = {\bf 1}_{[0, \, \zeta]}$, the function $\overline{G}_r^{(f)}$ in (\ref{Gbarf}) becomes
$$
\overline{G}_r^{(f)}(x)
=
\begin{cases} \E \Big[ \ee^{-\sum_{j=1}^n \alpha_j \int_{x+A_j} \d \mathscr{N}(r)  } \, {\bf 1}_{ \{ \min \mathscr{N}(r) \ge x\} } \Big], &\text{ if $r\in [0, \, \zeta]$}, \cr\cr
1- G_r(x), &\text{ if $r>\zeta$.}
\cr\end{cases}
$$

\noindent So, if we write
\begin{equation}
    G^*_v(x)
:=
1- \E \Big[ \ee^{-\sum_{j=1}^n \alpha_j \int_{x+A_j} \d \mathscr{N}(v)  } \, {\bf 1}_{ \{ \min \mathscr{N}(v) \ge x\} } \Big] ,
    \label{Gv*}
\end{equation}

\noindent then for $t\ge \zeta$, we have $\int_0^t [1-\overline{G}_v^{(f)}(\sigma B_v)] \d v = \int_0^\zeta G^*_v(\sigma B_v)\d v + \int_\zeta^t G_v(\sigma B_v)\d v$, so that by writing\footnote{Attention: $I_{(\ref{ht})}(t, \, \zeta)$ depends also on $y$.}
\begin{equation}
    I_{(\ref{ht})}(t, \, \zeta)
    :=
    t\, \E_{0,{y\over \sigma}}^{(t)} \Big[ {\bf 1}_{A(\sigma \overline B)} \,
    F_1(\sigma B_s, \, s\in [0, \, \zeta]) \,
    \ee^{- 2  \int_0^\zeta G^*_v(\sigma B_v)\d v - 2   \int_\zeta^t G_v(\sigma B_v)\d v}
    \Big] ,
    \label{ht}
\end{equation}

\noindent we have
\begin{eqnarray*}
 &&\E\Big\{ {\bf 1}_{A(X_{1,t})} \, F_1(Y_t(s), \, s\in [0, \, \zeta]) \,
    \ee^{-\sum_{j=1}^n \alpha_j L_j(t, \, \zeta)} \,
    F_2(X_1(t)-m_t)
    \Big\}
    \\
 &=&{1\over t}\int_\r \P \{ B_t \in {\! \d y\over \sigma} \} \, \ee^y \,
    F_2(y-m_t)\, I_{(\ref{ht})}(t, \, \zeta)
    \\
 &=& {1\over t^{3/2}}\int_\r {\ee^{y - {y^2\over 2 \sigma^2 t}} \over \sigmaÊ(2\pi )^{1/2}} \,
    F_2(y-m_t)\, I_{(\ref{ht})}(t, \, \zeta) \d y .
\end{eqnarray*}

\noindent Let $y:= z+m_t$. Since $F_2:={\bf 1}_{[\eta_1,\eta_2]}$, we have when $t\to \infty$, $\ee^{y - {y^2\over 2 \sigma^2 t}} \sim \ee^y = t^{3/2}\ee^{C_B}\ee^z$ where the numerical constant $C_B$ is in (\ref{bramson-constant}). Therefore, for $t\to \infty$,
\begin{eqnarray}
 &&\E\Big\{ {\bf 1}_{A(X_{1,t})} \, F_1(Y_t(s), \, s\in [0, \, \zeta]) \,
    \ee^{-\sum_{j=1}^n \alpha_j L_j(t, \, \zeta)} \,
    F_2(X_1(t)-m_t)
    \Big\}
    \nonumber
    \\
 &\sim&  \ee^{C_B}\int_{\eta_1}^{\eta_2} {\ee^z \over \sigmaÊ(2\pi)^{1/2}} \,
       I_{(\ref{ht})}(t, \, \zeta) \d z .
    \label{P(E)}
\end{eqnarray}

\noindent We need to treat $I_{(\ref{ht})}(t, \, \zeta)$ when $z\in[\eta_1,\eta_2]$.
As we will let $\zeta \to \infty$ before making $x\to\infty$, we can suppose $\zeta>x$. Let us write $\theta=\theta_B(\zeta) := \inf\{s\in [0,\zeta] : B_s = \max_{u \in [0,\zeta]} B_u \}.$ 
Applying the Markov property at time $v=\zeta$ (for the Brownian bridge which is an inhomogeneous Markov process, see Fact \ref{f:Markov bridge}), gives
\begin{eqnarray*}
 &&I_{(\ref{ht})}(t, \, \zeta)
    \\
 &=& t \int_{-\zeta^{2/3}}^{-\zeta^{1/3}}
    \E \Big[ {\bf 1}_{\{\max_{[0,\zeta]}\sigma B_s\le x,
    \, \theta \le x\}}
    F_1(\sigma B_s, \, s\in [0, \, \zeta]) \, \ee^{- 2  \int_0^\zeta G^*_v (\sigma B_v) \d v} \,
    {\bf 1}_{\{ \sigma B_\zeta \in \! \d w\} } \Big] \times
    \\
 && \qquad \times  \left(t\over t-\zeta\right)^{1/2} \frac{\ee^{-{(y-w)^2\over 2\sigma^2 (t-\zeta)}}}{\ee^{-{y^2\over 2 \sigma^2 t}}} \E_{0,{y- 
 w\over \sigma}}^{(t-\zeta)} \Big[{\bf 1}_{\{ \sigma \overline B_s \ge a_s,\, s\in[0,t-\zeta] \}} \ee^{- 2  \int_0^{t-\zeta} G_{v+\zeta}(w + \sigma B_v) \d v} \Big]
\end{eqnarray*}

\noindent where now $\overline B_s:=B_{t-\zeta}-B_{t-\zeta-s}$. We recall that we look at the case $z=y-m_t \in[\eta_1,\eta_2]$. It yields that ${(y-w)^2\over t-\zeta}$ and ${y^2\over t}$ are $o_t(1)$, so that, for $t\to \infty$,
\begin{eqnarray}\nonumber
 &&I_{(\ref{ht})}(t, \, \zeta)
    \\
 &\sim& t \int_{-\zeta^{2/3}}^{-\zeta^{1/3}}
    \E \Big[{\bf 1}_{\{\max_{[0,\zeta]} \sigma B_s\le x, \, \theta \le x\}}
    F_1(\sigma B_s, \, s\in [0, \, \zeta]) \, \ee^{- 2  \int_0^\zeta G^*_v (\sigma B_v) \d v} \,
    {\bf 1}_{\{ \sigma B_\zeta \in \! \d w\} } \Big] \times
    \nonumber
    \\
 && \qquad \times  \E_{0,{y- w\over \sigma}}^{(t-\zeta)} \Big[{\bf 1}_{\{\sigma \overline B_s \ge a_s, \, s\in[0,t-\zeta] \}} \ee^{- 2  \int_0^{t-\zeta} G_{v+\zeta}(w + \sigma B_v) \d v} \Big].
 \label{lim-I[0,t-zeta]}
\end{eqnarray}

\noindent At this stage, we need a couple of lemmas, stated as follows. We postpone the proof of these lemmas, and finish the proof of Theorem \ref{T: structure de Q}. Recalling the family of processes $\Gamma^{(b)}$ from (\ref{U}), we write
\begin{equation}
     \varphi_x(z)
     :=
   \sigma  \int_0^{x/\sigma}
    \E \Big[ \ee^{ -2  \int_0^\infty F_W({z}+\sigma \Gamma^{(b)}_v)\d v}\Big]
    \d b ,
    \qquad z\in \r,
    \label{phi}
\end{equation}
\noindent where $F_W$ is the distribution function of the random variable $W$ introduced in (\ref{bramson}).

\medskip

\begin{lemma}
\label{l:g2}
 Let $z\in \r$, $y:= z + m_t$, $x>0$ and $(a_s,s\in[0,t])$ defined in (\ref{def:as}). There exists a function $f:\r\times\r_+\to\r$ such that for any $w<x+z$ and $\zeta>0$ $$
 \lim_{t\to \infty} \,
 t \, \E_{0,{y-w\over \sigma}}^{(t)}
 \Big[{\bf 1}_{\{ \sigma(B_t - B_{t-s}) \ge a_s,s\in[0,t] \}}
 \ee^{- 2  \int_0^t G_{\zeta+v}(w+ \sigma B_v) \d v}
 \Big]
 = \varphi_x(z)f(w,\zeta).
 $$
 \noindent Moreover $f(w,\zeta)\sim |w|$ as $w\to-\infty$ and uniformly in $\zeta>0$.
\end{lemma}

\begin{lemma}
\label{l:queue}
 Let $\Gamma^{(b)}$ be the family of processes defined in $(\ref{U})$, and let $T_b := \inf\{ t\ge 0: \, B_t=b\}$.
 We have
 \begin{eqnarray*}
  &&\lim_{\zeta\to \infty}
     \E \Big[{\bf 1}_{\{\max_{[0,\zeta]} \sigma B_s\le x,\, \sigma B_\zeta\in (-\zeta^{2/3},-\zeta^{1/3}), \theta \le x\}}
     F_1(\sigma B_s, \, s\in [0, \, \zeta]) \,
     \ee^{- 2  \int_0^\zeta G^*_v (\sigma B_v) \d v} \, |B_\zeta|
     \Big]
     \\
  &=& \int_0^{x/\sigma}
     \E \Big[ F_1(\sigma \Gamma^{(b)}_s, \, s\ge 0)
     \ee^{-2  \int_0^\infty G^*_v(\sigma \Gamma^{(b)}_v) \d v}
     {\bf 1}_{\{ T_b \le x\} } \Big] \d b.
 \end{eqnarray*}

\end{lemma}

\medskip

\begin{remark}\label{rmk 6.3}
It is possible, with some extra work, to obtain the following identity.
\label{l:phi}
 Let $\varphi(z):=\lim_{x\to\infty} \varphi_x(z)$ be the limit of $(\ref{phi})$. Then for any $z\in \r$,
 $$
 \varphi(z)
 =
 {\sqrt{2 \pi}  \over c_1} \ee^{-(z+C_B)} \,
 f_{W}({z}),
 $$
 where $C_B$ is the constant in $(\ref{bramson-constant})$,
 $W$ the random variable in $(\ref{bramson})$,
 $f_{W}$ the density function of $W$, and
 $$
     c_1
     :=
     \int_0^\infty
     \E \Big[ \ee^{ -2  \int_0^\infty G_v(\sigma \Gamma^{(b)}_v)\d v}
     \Big] \d b ,
 $$
 with $\Gamma^{(b)}$ as defined in $(\ref{U})$. The appearance of $f_W$ here is due to the fact that standard arguments
 in the study of
 parabolic p.d.e.'s show that the density of $X_1(t)-m_t$ converges to that of $W.$ More precisely, using the classical interior parabolic a priori estimate \cite{friedman},  it is possible to show that $v(t,\cdot) \equiv u(t,m_t+\cdot)$ converges to $w(\cdot)$ in locally $C^2(\r)$ topology.
\end{remark}

\medskip
We now continue with the proof of Theorem \ref{T: structure de Q}. Let us go back to (\ref{lim-I[0,t-zeta]}). To apply Lemma \ref{l:g2}, we want to use dominated convergence. First, fix $\zeta>0$. Notice that
\begin{eqnarray*}
\E_{0,{y-
w\over \sigma}}^{(t-\zeta)} \Big[{\bf 1}_{\{ \sigma \overline B_s \ge a_s,s\in[0,t-\zeta] \}} \ee^{- 2  \int_0^{t-\zeta} G_{v+\zeta}(w + \sigma B_v) \d v} \Big]
&\le&
\P_{0,{y-
w\over \sigma}}^{(t-\zeta)} \Big(\sigma \overline B_s \ge a_s,s\in[0,t-\zeta]\Big)\\
&=&
\P_{0,{y-w\over \sigma}}^{(t-\zeta)} \Big( \sigma B_s \ge a_s,s\in[0,t-\zeta]\Big),
\end{eqnarray*}

\noindent the last identity being a consequence of the fact that $(\overline B_s, s\in[0,t-\zeta])$ and $(B_s, s\in[0,t-\zeta])$ have the same distribution under $\P_{0,{y-w\over \sigma}}^{(t-\zeta)}$. Using Lemma \ref{l:small o of 1/t} the last probability is smaller than ${c\over t-\zeta} x|z-w+x|$ for some constant $c>0$. Hence, we have for $t>2\zeta$
$$
t\,\E_{0,{y-
w \over \sigma}}^{(t-\zeta)} \Big[{\bf 1}_{\{\sigma  \overline B_s \ge a_s,s\in[0,t-\zeta] \}} \ee^{- 2  \int_0^{t-\zeta} G_{v+\zeta}(w + \sigma B_v) \d v} \Big]
\le
{c\over 2} \, x|z-w+x|.
$$

\noindent We check that
\begin{eqnarray*}
&& \int_{-\zeta^{2/3}}^{-\zeta^{1/3}} \E \Big[{\bf 1}_{\{\max_{[0,\zeta]} \sigma B_s\le x\}}
    F_1(\sigma B_s, \, s\in [0, \, \zeta]) \, \ee^{- 2  \int_0^\zeta G^*_v (\sigma B_v) \d v} \,
    {\bf 1}_{\{ \sigma B_\zeta \in \! \d w\} } \Big] |z-w+x|\\
& \le&  |\!| F_1|\!|_\infty \, \E [ \, |z - \sigma B_\zeta + x| \,
    ]
\end{eqnarray*}
which is finite. Hence, we can apply the dominated convergence, to see that
\begin{eqnarray*}
    \lim_{t\to\infty} I_{(\ref{ht})}(t, \, \zeta)
 &=& \varphi_x(z)\E \Big[
    {\bf 1}_{\{\max_{[0,\zeta]} \sigma B_s\le x, \,
    \sigma B_\zeta\in (-\zeta^{2/3},-\zeta^{1/3}), \, \theta \le x \}}
    \times
    \\
 &&\qquad\qquad\times
    F_1(\sigma B_s, \, s\in [0, \, \zeta]) \,
    \ee^{- 2  \int_0^\zeta G^*_v (\sigma B_v) \d v} \,
    f(\sigma B_\zeta,\zeta) \Big].
\end{eqnarray*}

\noindent Since $f(w,\zeta)\sim |w|$ when $w\to-\infty$ and uniformly in $\zeta>0$, we have as $\zeta\to\infty$,
\begin{eqnarray*}
    \lim_{t\to\infty} I_{(\ref{ht})}(t, \, \zeta)
 &\sim& \varphi_x(z)
    \E \Big[
    {\bf 1}_{\{\max_{[0,\zeta]} \sigma B_s\le x, \,
    \sigma B_\zeta \in (-\zeta^{2/3},-\zeta^{1/3}), \, \theta \le x\}}
    \times
    \\
 &&\qquad\qquad\times
    F_1(\sigma B_s, \, s\in [0, \, \zeta]) \,
    \ee^{- 2  \int_0^\zeta G^*_v (\sigma B_v) \d v} \,
    \sigma |B_\zeta| \Big] ,
\end{eqnarray*}

\noindent which, in view of Lemma \ref{l:queue}, gives that
$$
\lim_{\zeta\to\infty }\lim_{t\to \infty} I_{(\ref{ht})}(t, \, \zeta)
=
\varphi_x(z) \sigma\, \int_0^{x/\sigma}
\E \Big[ F_1(\sigma \Gamma^{(b)}_s, \, s\ge 0) \,
\ee^{-2  \int_0^\infty G^*_v(\sigma \Gamma^{(b)}_v) \d v}
{\bf 1}_{\{ T_b \le x\} } \Big] \d b.
$$

\noindent Going back to (\ref{P(E)}), this tells that
\begin{eqnarray*}
 &&\lim_{\zeta\to \infty} \lim_{t\to \infty}
    \E\Big\{ {\bf 1}_{A(X_{1,t})} \, F_1(Y_t(s), \, s\in [0, \, \zeta]) \,
    \ee^{-\sum_{j=1}^n \alpha_j L_j(t, \, \zeta)} \,
    F_2(X_1(t)- m_t)
    \Big\}
    \\
 &=&{\ee^{C_B}\over (2\pi)^{1/2}}
    \Big( \int_{\eta_1}^{\eta_2} \varphi_x(z) \ee^z \d z\Big)
    \Big( \int_0^{x/\sigma} \E \Big[ F_1(\sigma \Gamma^{(b)}_s, \, s\ge 0)\,
    \ee^{-2  \int_0^\infty G^*_v(\sigma \Gamma^{(b)}_v) \d v}
    {\bf 1}_{\{ T_b \le x\} } \Big] \d b\Big).
\end{eqnarray*}

\noindent Letting $x\to \infty$ yields that $\{(Y_t(s \in [0,t]) ; \sQ(t,\zeta)\}$ converges in distribution to $\{(Y(s), s\ge 0) ; \sQ\}$, that $X_{1}(t)-m_t$ converges in distribution, necessarily to $W$ (by (\ref{bramson})), and that $\{(Y_t(s \in [0,t]) ; \sQ(t,\zeta)\}$ and $X_{1}(t)-m_t$ are asymptotically independent. Theorem \ref{T: structure de Q} is proved.\hfill $\Box$

\bigskip

We observe that by letting $x \to \infty$ the last identity proves that 
\begin{equation}\label{E:c1}
\int_0^\infty \E \Big[ 
    \ee^{-2  \int_0^\infty G^*_v(\sigma \Gamma^{(b)}_v) \d v} \Big] \d b <\infty , \ \  \int_{-\infty}^{\infty}  \int_0^\infty 
    \E \Big[ \ee^{ -2  \int_0^\infty F_W({z}+\sigma \Gamma^{(b)}_v)\d v}\Big]   \ee^z \d b \d z<\infty.
\end{equation}     

\bigskip

It remains to check Lemmas \ref{l:g2} and \ref{l:queue}.
Their proof relies on some well known path decomposition results recalled in Section \ref{s:rappels-Denisov-Imhof}. Lemmas \ref{l:g2} and \ref{l:queue} are proved in Sections \ref{s:proof-lemma7.1} and \ref{s:1er-appli-Denisov}, respectively.

\medskip

Before proceeding with this program, observe that the arguments used above also yield the following Laplace transform characterization of $\sQ$.
For any $n \in \N$, $(\alpha_1, \ldots, \alpha_n) \in \r_+^n$ and $A_1,\ldots, A_n$ a collection of Borel subsets of $\r_+$ and $\zeta>0$  define
$$
I_\zeta(t) := \E\Big\{
    \exp\Big( -
    \sum_i
   \mathbf{1}_{\{ t-\tau_i(t) \le \zeta\} } \,
    \sum_{j=1}^n \alpha_j\, \int_{ A_j} \d \mathscr{N}_i(t) 
    \Big)
    \Big\} ,
$$
(i.e., only the particles whose common ancestor with $X_1(t)$ is more recent than $\zeta$ are taken into account). Clearly, the functional $I_\zeta(t)$ characterizes the law of
$\sQ(t,\, \zeta)$.

Then, for all $n$ and all bounded Borel sets $A_1$, $\cdots$, $A_n$ of $\r_+$, the Laplace transform of the distribution of the random vector $(\mathscr{Q}(A_1), \cdots , \mathscr{Q}(A_n))$ is given by:
$\forall \alpha_j \ge 0$ (for $1\le j\le n$),
\begin{eqnarray}
    \E\Big\{ \ee^{-\sum_{j=1}^n \alpha_j \mathscr{Q}(A_j)}
    \Big\}
    &=&  \lim_{\zeta \to \infty} \lim_{t \to \infty} I_{\zeta}(t) \nonumber \\&=&
    {\int_0^\infty \E(\ee^{-2  \int_0^\infty
    G_v^*(\sigma \Gamma^{(b)}_v) \d v}) \d b
    \over
    \int_0^\infty \E(\ee^{-2  \int_0^\infty
    G_v (\sigma \Gamma^{(b)}_v) \d v}) \d b} ,
    \label{Q-process}
\end{eqnarray}

\noindent where
\begin{eqnarray*}
   G_v^*(x)
 &:=& 1-
    \E \Big[ \ee^{-\sum_{j=1}^n \alpha_j
    \int_{x+A_j} \d \mathscr{N}(v)  } \,
    {\bf 1}_{ \{ \min \mathscr{N}(v) \ge x\} } \Big] .
\end{eqnarray*}
Observe that the first equality in (\ref{Q-process}) is a consequence of the convergence in distribution of $\sQ(\zeta,t)$ given in Theorem \ref{T: structure de Q}.

\section{Meander, bridge and their sample paths}
\label{s:rappels-Denisov-Imhof}

We collect in this section a few known results of Brownian motion and related processes. Recall that if $g:= \sup\{ t<1: \, B_t = 1\}$, then  $(\mathfrak{m}_u := (1-g)^{-1/2} |B_{g+ (1-g)u}|, \, u\in [0, \, 1])$ is called a Brownian meander. In particular, $\mathfrak{m}_1$ has the Rayleigh distribution: $\P(\mathfrak{m}_1 >x) = \ee^{-x^2/2}$, $x>0$.

Let $B$ be Brownian motion, $R$ a three-dimensional Bessel process, and $\mathfrak{m}$ a Brownian meander. The processes $B$ and $R$ are assumed to start from $a$ under $\P_a$ (for $a\ge 0$) if stated explicitly; otherwise we work under $\P:= \P_0$ so that they start from 0.

\begin{fact}
\label{f:denisov}

 {\bf (Denisov~\cite{denisov})}
 Let $\theta := \inf\{ s\ge 0: \, B_s = \sup_{u\in [0, \, 1]} B_u\}$
 be the location of the maximum of $B$ on $[0, \, 1]$.
 The random variable $\theta$ has the Arcsine law:
 $\P(\theta \le x) = {2\over \pi} \arcsin \sqrt{x}$,
 $x\in [0, \, 1]$.
 The processes
 $({B_\theta - B_{(1-u)\theta}\over \theta^{1/2}}, \, u\in [0, \, 1])$
 and
 $({B_\theta - B_{\theta + u(1-\theta)}\over (1-\theta)^{1/2}}, \, u\in [0, \, 1])$
 are independent copies of the Brownian meander,
 and are also independent of the random variable $\theta$.

\end{fact}

\begin{fact}
\label{f:imhof}

 {\bf (Imhof~\cite{imhof})}
 For any continuous function
 $F : \, C([0, \, 1], \, \r) \to \r_+$, we have
 $$
 \E\Big[ F(\mathfrak{m}_s, \; s\in [0, \, 1]) \Big]
 =
 \Big( {\pi \over 2}\Big)^{\! 1/2} \,
 \E\Big[ {1\over R_1} \, F(R_s, \; s\in [0, \, 1]) \Big] .
 $$
 In particular, for any $x>0$, the law of
 $(\mathfrak{m}_s, \; s\in [0, \, 1])$ given $\mathfrak{m}_1=x$ is the law of
 $(R_s, \; s\in [0, \, 1])$ given $R_1=x$.

\end{fact}

\begin{corollary}
\label{c:imhof}

 Let $r>0$ and $q>0$.
 Let $T_a := \inf\{ s\ge 0: \, B_s=a\}$
 for any $a\in \r$.

 {\rm (i)} The law of
 $(\mathfrak{m}_1-\mathfrak{m}_{1-s}, \; s\in [0, \, 1])$ under
 $\P( \, \bullet \, | \, \mathfrak{m}_1=r)$
 is the law of
 $(q^{-1/2}B_{qs}, \; s\in [0, \, {1\over q}T_{q^{1/2}r}])$ under
 $\P( \, \bullet \, | \, T_{q^{1/2}r}=q)$.

 {\rm (ii)} For any $t>0$, the law of
 $(R_1-R_{1-s}, \; s\in [0, \, 1])$ under
 $\P( \, \bullet \, | \, R_1=r)$
 is the law of $(q^{-1/2}(B_0-B_{qs}), \, s\in [0, \, {T_0\over q}])$
 under $\P_{q^{1/2}r}( \, \bullet \, | \, T_0=q)$.

\end{corollary}

\noindent {\it Proof.} By Imhof's theorem (Fact \ref{f:imhof}), $(\mathfrak{m}_s, \; s\in [0, \, 1])$ given $\mathfrak{m}_1=r$, as well as $(R_s, \; s\in [0, \, 1])$ given $R_1=r$, are three-dimensional Bessel bridges of length $1$, starting from $0$ and ending at $r$. By Williams~\cite{williams}, this is equivalent to saying that both $(\mathfrak{m}_1-\mathfrak{m}_{1-s}, \; s\in [0, \, 1])$ given $\mathfrak{m}_1=r$, and $(R_1-R_{1-s}, \; s\in [0, \, 1])$ given $R_1=r$, have the distribution of $(B_s, \; s\in [0, \, T_r])$ given $T_r=1$.

By scaling, this gives (i).

To get (ii), we use moreover the fact that, by symmetry, $(B_s, \; s\in [0, \, T_r])$ under $\P( \, \bullet \, | \, T_r=1)$ has the law of $(-B_s, \; s\in [0, \, T_{-r}])$ under $\P( \, \bullet \, | \, T_{-r}=1)$, and thus has the law of $(B_0-B_s, \; s\in [0, \, T_0])$ under $\P_r( \, \bullet \, | \, T_0=1)$. This yields (ii) by scaling.\hfill$\Box$

Finally, we will use several times the Markov property for the Brownian bridge which is an inhomogeneous Markov process. Recall that $\E^{(t+s)}_{0,x}$ is expectation with respect to $\P^{(t+s)}_{0,x} (\, \cdot \, ) := \P_0 ( \, \cdot \, | \, B_{t+s}=x)$.

\begin{fact}
\label{f:Markov bridge}

 Fix $t$, $s \ge 0$ and $x \in \r$.
 For any measurable functions
 $F : \, C([0, \,t], \, \r) \to \r_+$ and
 $G : \, C([0, \,s], \, \r) \to \r_+$, we have
 \begin{align*}
& \E^{(t+s)}_{0,x} \Big[ F( B_s, s\in[0,t]) G(B_r, r\in[t,t+s])  \Big] \\
 &=  \E_0 \Big[  \sqrt{ {t+s\over s}}\ee^{{x^2\over 2(t+s)}-{(x-B-t)^2\over 2s}} F( B_s, s\in[0,t])    \E^{(s)}_{B_t, x} \left\{G(B_r, r\in[t,t+s]) \right\} \Big] .
\end{align*}
\end{fact}

\section{Proof of Lemma \ref{l:queue}}
\label{s:1er-appli-Denisov}

Let $x>0$ and let $F_1: \, C(\r_+, \, \r) \to \r_+$  be a bounded continuous function. We need to check 
\begin{eqnarray*}
 &&\lim_{\zeta\to \infty}
    \E \Big[{\bf 1}_{\{\max_{[0,\zeta]} \sigma B_s\le x,\sigma B_\zeta\in (-\zeta^{2/3},-\zeta^{1/3}), \theta \le x\}}
    F_1(\sigma B_s, \, s\in [0, \, \zeta]) \,
    \ee^{- 2  \int_0^\zeta G^*_v (\sigma B_v) \d v} \, |B_\zeta|
    \Big]
    \\
 &=& \int_0^{x/\sigma}
    \E \Big[ F_1(\Gamma^{(b)}_s, \, s\ge 0) \,
    \ee^{-2  \int_0^\infty G^*_v(\sigma \Gamma^{(b)}_v) \d v}
    \Big] \d b,
\end{eqnarray*}

\noindent where $\Gamma^{(b)}$ is the process defined in (\ref{U}), $\theta=\theta_B(\zeta) := \inf\{s\in [0,\zeta] : B_s = \max_{u \in [0,\zeta]} B_u \}$, and $G_v(\cdot)$ is the function defined in (\ref{Gv*}). [We do not use any particular property of $G_v$ except its measurability and positivity.]

The random variable ${\theta \over \zeta}$ has the Arcsine law. According to Denisov's theorem (Fact \ref{f:denisov}), the two processes\footnote{The processes $Y$ and $Z$ depend, of course, on $\zeta$.} $(Y_u := {B_\theta - B_{(1-u)\theta}\over \theta^{1/2}}, \, u\in [0, \, 1])$ and $(Z_u := {B_\theta - B_{\theta + u(\zeta-\theta)}\over (\zeta-\theta)^{1/2}}, \, u\in [0, \, 1])$ are independent Brownian meanders, and are also independent of the random variable $\theta$.

By definition,
\begin{eqnarray}
    \int_0^\zeta G^*_v (\sigmaÊB_v) \d v
 &=&\theta \int_0^1 G^*_{u \theta} (\sigma \theta (Y_1 -Y_{1-u})) \d u +
    \nonumber
    \\
 && + (\zeta-\theta) \int_0^1 G^*_{\theta + u(\zeta-\theta)} ( \sigma \theta^{1/2} Y_1 - \sigma (\zeta-\theta)^{1/2} Z_u) \d u.
    \label{meander1}
\end{eqnarray}

\noindent Also, $B_\zeta = \theta^{1/2} Y_1 - (\zeta-\theta)^{1/2} Z_1$, and
\begin{equation}
    B_s
    =
    \begin{cases}
     \theta^{1/2} (Y_1 - Y_{1- {s\over \theta}}),
      &\text{ if $s\in [0, \, \theta]$,} \cr\cr
     \theta^{1/2} Y_1 - (\zeta -\theta)^{1/2} Z_{s- \theta\over \zeta- \theta},
      &\text{ if $s\in [\theta, \, \zeta]$.} \cr
    \end{cases}
    \label{meander2}
\end{equation}

\medskip

\begin{lemma}
\label{l:meandre=>Bessel}

 Let $(\mathfrak{m}_s, \, s\in [0, \, 1])$ be a Brownian meander. Let $\varepsilon^1: \, \r_+ \to \r_+$ and $\varepsilon^2: \, \r_+ \to \r_+$ be two measurable functions such that $\lim_{t\to\infty} \varepsilon^1_t=0$ and $\lim_{t\to\infty} \varepsilon_t^2=\infty$.
 For $x\in \r$, $\ell \in \r$, $a\ge 0$, $b\ge 0$
 and bounded continuous function
 $F : \, C([0, \, 1], \, \r) \to \r_+$, we have
 \begin{eqnarray*}
  &&\lim_{t\to \infty}
     \E\Big[{\bf 1}_{\{\mathfrak{m}_1\in (\varepsilon_t^1,\varepsilon^2_t)\}}
     \mathfrak{m}_1\, F(t^{1/2}\mathfrak{m}_{bs\over t}, \, s\in [0, \, 1])\,
     \ee^{-at\int_0^1 G^*_{x+ ut} (\ell - \sigma t^{1/2} \mathfrak{m}_u) \d u}
     \Big]
     \\
  &=& \Big({\pi \over 2}\Big)^{\! 1/2} \,
     \E\Big[
     F(R_{bs}, \, s\in [0, \, 1])\,
     \ee^{-a\int_0^\infty G^*_{x+v} (\ell - \sigma R_v)  \d v}
     \Big] ,
 \end{eqnarray*}
 where $R$  is a
 three-dimensional Bessel process.
\end{lemma}

\medskip

\noindent {\it Proof of Lemma \ref{l:meandre=>Bessel}.} By Imhof's theorem (Fact \ref{f:imhof}), we have, for $t\ge b$,
\begin{eqnarray*}
 &&\E\Big[{\bf 1}_{\{\mathfrak{m}_1\in(\varepsilon_t^1,\varepsilon^2_t)\}}
     \mathfrak{m}_1\, F(t^{1/2}\mathfrak{m}_{bs\over t}, \, s\in [0, \, 1])\,
     \ee^{-at\int_0^1 G^*_{x+ ut} (\ell - \sigma t^{1/2} \mathfrak{m}_u) \d u}
     \Big]
     \\
 &=& \Big({\pi \over 2}\Big)^{\! 1/2} \,
    \E\Big[{\bf 1}_{\{R_1\in(\varepsilon_t^1,\varepsilon^2_t)\}} F(t^{1/2}R_{bs\over t}, \, s\in [0, \, 1])\,
    \ee^{-at\int_0^1 G^*_{x+ ut} (\ell - \sigma t^{1/2} R_u) \d u}
    \Big]
    \\
 &=& \Big({\pi \over 2}\Big)^{\! 1/2} \,
    \E\Big[{\bf 1}_{\{R_tt^{-1/2}\in(\varepsilon_t^1,\varepsilon^2_t)\}} F(R_{bs}, \, s\in [0, \, 1])\,
    \ee^{-a\int_0^t G^*_{x+v} (\ell - \sigma R_v) \d v} \Big] ,
\end{eqnarray*}

\noindent the second identity being a consequence of the scaling property. Let $t\to \infty$. Since $\P(R_tt^{-1/2}\notin(\varepsilon_t^1,\varepsilon^2_t))\to 0$, Lemma \ref{l:meandre=>Bessel} follows by dominated convergence.\hfill$\Box$

\bigskip

\noindent {\it Proof of Lemma \ref{l:queue}.} Recall (\ref{meander1}) and (\ref{meander2}). Let $F_{1,a}(Y,Z) := F_1(  a^{1/2} \sigma (Y_1 - Y_{1- {s\over a}}){\bf 1}_{\{ s\le a \}} +\sigma (a^{1/2} Y_1 - (\zeta -a)^{1/2} Z_{s- a\over \zeta- a} ){\bf 1}_{\{ s\ge a \}}, \ s\in[0,\zeta])$. Then
\begin{eqnarray*}
 &&
\E \Big[ {\bf 1}_{\{\max_{[0,\zeta]}  \sigma B_s\le x, \sigma B_\zeta\in ( -\zeta^{2/3},-\zeta^{1/3}), \theta \le x\}}
F_1(\sigma B_s, \, s\in [0, \, \zeta]) \,
\ee^{- 2  \int_0^\zeta G^*_v (\sigma B_v) \d v} \, |B_\zeta| \Big] \\
&=& \int_0^x \P(\theta \in \! \d a)\, \E \Big[  {\bf 1}_{\{ \sigma a^{1/2}Y_1 \le x\}}  F_{1,a}(Y, Z) \ee^{-2  a \int_0^1 G^*_{av}(\sigma a^{1/2} (Y_1 - Y_{1-v}))\d v} \\ && \qquad \ee^{-2(\zeta-a)\int_0^1 G^*_{a+v(\zeta-a)} (\sigma a^{1/2} Y_1  -(\zeta-a)^{1/2} \sigma Z_v)\d v}  | a^{1/2} Y_1-(\zeta-a)^{1/2} Z_1 | {\bf 1}_{\{  -\sigma B_\zeta \in  [\zeta^{1/3},\zeta^{2/3}  ] \}} \Big]
\\
&=& \int_0^x \zeta^{1/2} \P(\theta \in \! \d a)\, \E \Big\{  {\bf 1}_{\{ \sigma a^{1/2}Y_1 \le x\}}  \ee^{-2  a \int_0^1 G^*_{av}(\sigma a^{1/2} (Y_1 - Y_{1-v}))\d v} \\ && \qquad \qquad \E \Big[  F_{1,a}(Y, Z)  \ee^{-2  (\zeta-a) \int_0^1 G^*_{a+v(\zeta-a)} (\sigma a^{1/2} Y_1  -(\zeta-a)^{1/2} \sigma Z_v)\d v} {| a^{1/2} Y_1-(\zeta-a)^{1/2} Z_1|\over \zeta^{1/2}} \\
&& \qquad \qquad \qquad {\bf 1}_{\{ Z_1  \in  [\varepsilon_\zeta^1,\varepsilon_\zeta^2 ] \}} \, \big| \, Y_s, \, s\le 1\Big] \Big\} ,
\end{eqnarray*}

\noindent where $\varepsilon_\zeta^1:= ({\zeta^{1/3}\over \sigma} +a^{1/2} Y_1)(\zeta-a)^{-1/2} $ and $\varepsilon_\zeta^2 := ({\zeta^{2/3}\over \sigma} +a^{1/2} Y_1)(\zeta-a)^{-1/2}$.

By Lemma \ref{l:meandre=>Bessel}, we get that for each $a \in [0,x]$ when $\zeta \to \infty$, the conditional expectation $\E [ \, \ldots \, | \, Y_s, \, s\le 1]$ on the right-hand side converges to
$$
\Big({\pi \over 2} \Big)^{1/2} \E \Big[ \bar F_{1,a}(Y,R) \ee^{- 2  \int_0^\infty G^*_{v+a}( \sigma a^{1/2} Y_1-\sigma R_v )\d v }  \big| Y_s, s\le 1\Big]
$$
where
$$
\bar F_{1,a}(Y,R) := F_1(  \sigma  a^{1/2} (Y_1 - Y_{1- {s\over a}}){\bf 1}_{\{ s\le a \}} + \sigma (a^{1/2} Y_1 -  R_{s-a} ){\bf 1}_{\{ s\ge a \}}, \ s\in[0,\infty)) ,
$$

\noindent with $R$ and $Y$ being independent.
Since we only allow $a$ to vary between 0 and $x$ we may conclude that
\begin{eqnarray}
 &&\lim_{\zeta\to \infty}
\E \Big[ {\bf 1}_{\{\max_{[0,\zeta]}  \sigma B_s\le x, \sigma B_\zeta\in ( -\zeta^{2/3},-\zeta^{1/3}), \theta \le x\}}
F_1(\sigma B_s, \, s\in [0, \, \zeta]) \,
\ee^{- 2  \int_0^\zeta G^*_v (\sigma B_v) \d v} \, |B_\zeta| \Big]
    \nonumber
    \\
 & & =\int_0^x {\! \d a \over (2\pi a)^{1/2}} \,
    \E \Big[ {\bf 1}_{\{ \sigma a^{1/2} \mathfrak{m}_1\le x\}}
    \bar F_{1,a}(\mathfrak{m},R) \times
    \nonumber
    \\
 && \qquad\times
    \ee^{-2  a \int_0^1 G^*_{au} (\sigmaÊa^{1/2} (\mathfrak{m}_1 - \mathfrak{m}_{1-u})) \d u
    -2  \int_0^\infty G^*_{a+v} (\sigma a^{1/2} \mathfrak{m}_1 - \sigma R_v) \d v} \Big]
    \nonumber
    \\
 && =:I_{(\ref{c3})} ,
    \label{c3}
\end{eqnarray}

\noindent where the Brownian meander $\mathfrak{m}$ and the three-dimensional Bessel process $R$ are assumed to be independent. Let $V^{(a)}_s := a^{1/2}(\mathfrak{m}_1-\mathfrak{m}_{1-{s\over a}})$ if $s\in [0, \, a]$ and $V^{(a)}_s := a^{1/2} \mathfrak{m}_1- R_{s-a}$ if $s\ge a$. 
We observe that $a \int_0^1 G^*_{ua} (\sigma a^{1/2} (\mathfrak{m}_1 - \mathfrak{m}_{1-u})) \d u +\int_0^\infty G^*_{a+v} (\sigma a^{1/2} \mathfrak{m}_1 - \sigma R_v) \d v$ is, in fact, $\int_0^\infty G^*_s(\sigma V^{(a)}_s) \d s$. So
\begin{eqnarray*}
    I_{(\ref{c3})}
 &=& \int_0^x {\! \d a\over (2\pi a)^{1/2}} \,
    \E \Big[ {\bf 1}_{\{\sigma a^{1/2} \mathfrak{m}_1\le x\}}
    F_1(\sigma V^{(a)}_s, \, s\ge 0)
    \, \ee^{-2  \int_0^\infty G^*_s(\sigma V^{(a)}_s) \d s} \Big]
    \\
 &=&\int_0^x {\! \d a\over (2\pi a)^{1/2}}
    \int_0^{{x\over \sigma \sqrt{a}}} \!\! \d r \; r \ee^{-r^2/2} \times
    \\
 &&\qquad \times
    \E \Big[ F_1(\sigma V^{(a)}_s, \, s\ge 0)
    \, \ee^{-2  \int_0^\infty G^*_s(\sigma V^{(a)}_s) \d s}
    \, \Big| \, \mathfrak{m}_1=r \Big] ,
\end{eqnarray*}

\noindent where, in the last identity, we used the fact that $\mathfrak{m}_1$ has the Rayleigh distribution. Applying Corollary \ref{c:imhof}~(i) to $q:=a$, and recalling the process $\Gamma^{(a^{1/2}r)}$ from (\ref{U}), this yields
\begin{eqnarray*}
    I_{(\ref{c3})}
 &=& \int_0^x {\! \d a\over (2\pi a)^{1/2}}
    \int_0^{{x\over \sigma \sqrt{a}}} \!\! \d r \; r \ee^{-r^2/2} \times
    \\
 &&\qquad \times
    \E \Big[
    F_1( \sigma  \Gamma^{(a^{1/2}r)}_s, \, s\ge 0) \,
   \ee^{-2  \int_0^\infty G^*_v(\sigma \Gamma^{(a^{1/2}r)}_v) \d v}
    \, \Big| \, T_{a^{1/2}r}=a \Big] .
\end{eqnarray*}

\noindent By a change of variables $r:= a^{-1/2} b$ and Fubini's theorem, the expression on the right-hand is
\begin{eqnarray*}
 &=& \int_0^{x/\sigma} \!\! \d b \int_0^x \!\! \d a \,
    {b\ee^{-b^2/(2a)}\over (2\pi a^3)^{1/2}}
    \E \Big[ F_1( \sigma \Gamma^{(b)}_s, \, s\ge 0) \,
    \ee^{-2  \int_0^\infty G^*_v(\sigma \Gamma^{(b)}_v) \d v}
    \, \Big| \, T_b=a \Big]
    \\
 &=&\int_0^{x/\sigma} \!\! \d b \,
    \E \Big[ F_1( \sigma \Gamma^{(b)}_s, \, s\ge 0) \,
    \ee^{-2  \int_0^\infty G^*_v(\sigma \Gamma^{(b)}_v) \d v}
    {\bf 1}_{ \{ T_b \le x\} }\Big] ,
\end{eqnarray*}

\noindent completing the proof of Lemma \ref{l:queue}.\hfill$\Box$

\section{Proof of Lemma \ref{l:g2}}
\label{s:proof-lemma7.1}

We first recall the following fact concerning the F-KPP equation. As already pointed out, $u(t,x) := G_t(x)$ is the solution of a version of the F-KPP equation with heavyside initial data. Define $m_t(\varepsilon) := \inf \{ x: G_t(x) = \varepsilon \}$ for $\varepsilon \in (0, 1)$. 
Bramson \cite{bramson83} shows that, for any $\varepsilon\in (0, 1)$, there exists a constant $C(\varepsilon) \in \r$ such that $m_t(\varepsilon) = \frac3{2} \log t + C(\varepsilon) + o(1)$, $t\to \infty$.


\begin{fact}
\label{f:stretching-lemma}

 {\bf (McKean~\cite[pp.~326--327]{mckean})}
 For any $\varepsilon \in (0,1)$, let $c_\varepsilon = w^{-1}(\epsilon)$, i.e., $\P(W\le c_\varepsilon) = \varepsilon$. Then, for $\varepsilon \in (0,1)$, the following convergences are monotone as $t\to \infty$:
\begin{align*}
G_t(x+m_t(\varepsilon)) \nearrow \P( W \le x+c_\varepsilon ) &= w(x+c_\varepsilon ) \qquad \text{for } x \le 0,\\
G_t(x+m_t(\varepsilon)) \searrow \P(W \le x+c_\varepsilon ) &= w(x+c_\varepsilon ) \qquad \text{for } x \ge 0.
\end{align*}

\end{fact}


Recall that $G_t(m_t +x) \to w(x)$, $\forall x\in \r$, and that $m_t := {3\over 2}\log t + C_B$. Since $\P(W\le y) \sim C|y|\ee^y$, $y\to -\infty$ (see (\ref{lalley-sellke})), a consequence of Fact \ref{f:stretching-lemma} (in the case $x\le 0$) is that for some constant $c>0$, and any $v>0$ and $r\in \r$,
\begin{equation}
    G_v(m_v +r )\le c\, (|r|+1) \ee^{r}.
    \label{eq:stretching}
\end{equation}


Let us turn to the proof of Lemma \ref{l:g2}.
Let $B$ be Brownian motion (under $\P= \P_0$). Recall that $\E_{0,{y\over \sigma}}^{(t)}$ is expectation with respect to $\P_{0,{y\over \sigma}}^{(t)} := \P( \, \bullet \, | \, B_t = {y\over \sigma})$.
We further subdivise the proof of Lemma \ref{l:g2} into two lemmas.

\begin{lemma}\label{l:conv-kappa}
Let $\kappa:\r_+\to\r$ be a bounded Borel function with compact support. Take $x>0$ and recall the definition of $(a_s,s\in[0,t])$ in (\ref{def:as}). Then, for any $b>{a_0\over \sigma}$,
$$
\lim_{t\to\infty} t^{3/2} \,
\E_b\left[ {\bf 1}_{\{ \sigma B_s\ge a_s,s\in[0,t] \}}\, \kappa( \sigma B_{t}-a_t) \right]
=
{\sigma b - a_0\over 2\sqrt{\pi}} \int_{\r_+} r\kappa(r) \d r.
$$
\end{lemma}

\begin{lemma}\label{l:bessel decomposition}

 Let $F_W$ be the distribution function of $W$,
 where $W$ is the random variable in (\ref{bramson}).
 For any $z\in \r$,
 \begin{eqnarray*}
  &&\lim_{M\to\infty}
     \E\Big[
     {\bf 1}_{\{\sigma B_{s}\ge -x,\,s\in[0,M]\}}
     \ee^{- 2  \int_0^M F_W(z - \sigma B_v) \d v}
     (x+\sigma B_M)\Big]
     \\
  &=& x\, \E_{{x\over \sigma}}
     \Big[
     \ee^{- 2  \int_0^\infty F_W (z + x-\sigma R_v) \d v}
     \Big]
     \\
  &=& \varphi_x(z) ,
 \end{eqnarray*}
 with the notation of (\ref{phi}),
 and where $(R_v)_{v\ge 0}$ is a three-dimensional Bessel process.
\end{lemma}

Before proving Lemmas \ref{l:conv-kappa} and \ref{l:bessel decomposition}, let us see how we use them to prove Lemma  \ref{l:g2}.

\medskip

\noindent{\it Proof of Lemma \ref{l:g2}.} Recall that $y=z+m_t$ and $(a_s,s\in[0,t])$ is defined in (\ref{def:as}). Take $\zeta>0$ and $w<x+z$ where $x=-a_0$. Let
$$
h_v(r):= G_{\zeta+v}(w + r),
\qquad v\ge 0, \; r\in \r.
$$

\noindent So if we write
\begin{eqnarray}
    I_{(\ref{l:7.1-proof})}
 &:=&t \, \E_{0,{y-w\over \sigma}}^{(t)}
    \Big[{\bf 1}_{\{ \sigma(B_t - B_{t-s}) \ge a_s,\; s\in[0,t] \}}
    \, \ee^{- 2  \int_0^t G_{\zeta+v}(w+\sigma B_v) \d v}
    \Big]
    \nonumber
    \\
 &=&t \, \E_{0,{y-w\over \sigma}}^{(t)}
    \Big[{\bf 1}_{\{ \sigma(B_t - B_{t-s}) \ge a_s,\; s\in[0,t] \}}
    \, \ee^{- 2  \int_0^t h_v(\sigma B_v) \d v}
    \Big] ,
    \label{l:7.1-proof}
\end{eqnarray}

\noindent then we need to check that $\lim_{t\to \infty} I_{(\ref{l:7.1-proof})} = \varphi_x(z) f(w, \, \zeta)$ for some $f(w, \, \zeta)$ such that $f(w, \, \zeta) \sim |w|$ as $w\to -\infty$ and uniformly in $\zeta>0$.

Since $(B_t-B_{t-s}, \, s\in[0,t])$ and $(B_s, \, s\in[0,t])$ have the same distribution under $\P_{0,{y-w\over \sigma}}^{(t)}$, we have
$$
I_{(\ref{l:7.1-proof})}
=
t\, \E_{0,{y-w\over \sigma}}^{(t)}
\Big[{\bf 1}_{\{\sigma B_{s}\ge a_s,\,s\in[0,t]\}} \,
\ee^{- 2  \int_0^t h_{t-v}(y-w- \sigma B_v) \d v}
\Big].
$$


Recall from (\ref{eq:stretching}) that
$$
G_v(m_v +r )\le c\, (|r|+1) \ee^{r},
$$

\noindent for some constant $c>0$, and any $v>0$ and $r\in \r$.  Therefore, there exists a constant $c_{x,z}$, depending on $(x, \, z)$, such that $h_v(m_v+r) \le c_{x,z} (|r|+1)\ee^r$. Thus, on the event $\{ \sigma B_s> \min(s^{1/3},\, m_t +(t-s)^{1/3}), \, \forall  s\in[M,t-M] \}$, we have $\int_{M}^{t-M}h_{t-v}(y - B_v) \d v  \le \varepsilon(M)$ for any $t>1$, where $\varepsilon(M)$ is deterministic and statisfies $\lim_{M\to\infty} \varepsilon(M)=0$.

On the other hand recall from Lemma \ref{l:small o of 1/t} that
$$
\P_{0,{y-w\over \sigma}}^{(t)}\Big( \sigma B_{s}\ge a_s, \, s\in[0,t], \; \exists s\in[M,t-M]: \, \sigma B_s< \min(s^{1/3}, m_t +(t-s)^{1/3})\Big)
=
{1\over t} o_M(1),
$$

\noindent in the sense that $\limsup_{t\to\infty} t\P_{0,{y-w\over \sigma}}^{(t)}(\ldots)=o_M(1)$, where, as before, $o_M(1)$ designates an expression which converges to 0 as $M \to \infty$.
Therefore, we see that
\begin{eqnarray*}
    \lim_{t \to \infty} I_{(\ref{l:7.1-proof})}
 &=& \lim_{M \to \infty} \lim_{t\to \infty}
    t \, \E_{0,{y-w\over \sigma}}^{(t)}
    \Big[ {\bf 1}_{\{\sigma B_{s}\ge a_s,\; s\in[0,t]\}} \,
    {\bf 1}_{\{ \sigma B_{t-M}-a_t \in[M^{1/3},\, M^{2/3}]\}} \times
    \\
 &&\qquad\qquad\qquad \times
    \ee^{- 2 \int_{[0,M]\cup[t-M,t]} h_{t-v}(y-w - \sigma B_v) \d v}
    \Big].
\end{eqnarray*}

\noindent Define
\begin{eqnarray}
    \kappa_M(r)
 &:=&
    {\bf 1}_{\{ r\in[M^{1/3}, \, M^{2/3}]\}} \,
    \ee^{-{(x+z-w-r)^2\over 2\sigma^2M}} \times
    \nonumber
    \\
 && \qquad \times
    \E_{{r\over \sigma},{x+z-w\over \sigma}}^{(M)}
    \Big[ {\bf 1}_{\{ \min_{[0,M]} B>0 \}}\,
    \ee^{- 2 \int_0^M h_{M-v}(x+z-w - \sigma B_v) \d v} \Big] .
    \label{def:kappa}
\end{eqnarray}

\noindent By the Markov property (applied at time $t-M$, and then at time $M$ for the second identity), we get, for $t\to\infty$,
\begin{eqnarray}
 && t\, \E_{0,{y-w\over \sigma}}^{(t)}
    \Big[ {\bf 1}_{\{\sigma B_{s}\ge a_s,\, s\in[0,\, t]\}} \,
    {\bf 1}_{\{\sigma B_{t-M}-a_t \in [M^{1/3},\, M^{2/3}]\}} \,
    \ee^{- 2 \int_{[0,M]\cup[t-M,t]} h_{t-v}(y-w - \sigma B_v) \d v}
    \Big]
    \nonumber
    \\
 &\sim&{t^{3/2} \over M^{1/2}} \,
    \E_0 \Big[{\bf 1}_{\{\sigma B_{s}\ge a_s,\, s\in[0,\, t-M]\}} \,
    \ee^{- 2  \int_0^M h_{t-v}(y-w - \sigma B_v) \d v}
    \kappa_M(\sigma B_{t-M}-a_t) \Big]
    \nonumber
    \\
 &=&{t^{3/2} \over M^{1/2}} \,
    \E_0 \Big[{\bf 1}_{\{\sigma B_{s}\ge a_s,\, s\in[0,\,M]\}} \, 
    \ee^{- 2  \int_0^M h_{t-v}(y-w - \sigma B_v) \d v} \times
    \nonumber
    \\
 &&\qquad\qquad \times
    \E_{B_M}\Big(
    {\bf 1}_{\{\sigma B_{s}\ge a_{M+s},\, s\in[0, \, t-2M]\}}
    \kappa_M(\sigma B_{t-2M}-a_t) \Big) \Big] .
    \label{eq:Mt-M}
\end{eqnarray}

\noindent By Lemma \ref{l:conv-kappa}, almost surely,
$$
\lim_{t\to\infty}t^{3/2} \, \E_{B_M} \Big( {\bf 1}_{\{\sigma B_{s}\ge a_{M+s},\, s\in[0,t-2M]\}} \kappa_M(\sigma B_{t-2M}-a_t) \Big)
=
{x+ \sigma B_M \over 2\sqrt{\pi}} \int_{\r_+} r\kappa_M(r) \d r.
$$

\noindent On the other hand, $h_s(m_s + r)=G_{\zeta+s}(m_s+w+r) \to F_W(w+r)$ as $s\to \infty$ (see (\ref{bramson})). Hence, almost surely,
$$
\lim_{t\to\infty} \ee^{- 2 \int_0^M h_{t-v}(y-w - \sigma B_v) \d v}
=
\ee^{- 2  \int_0^M F_W(z - \sigma B_v) \d v}.
$$

\noindent In view of the Brownian motion sample path probability bound given in (\ref{conv-dominee}), below, we are entitled to use dominated convergence to take the limit $t\to\infty$ in (\ref{eq:Mt-M}):
\begin{eqnarray*}
 &&\lim_{t\to\infty} t\, \E_{0,{y-w\over \sigma}}^{(t)}
    \Big[{\bf 1}_{\{ \sigma B_{s}\ge a_s,\,s\in[0,t]\} } \,
    {\bf 1}_{\{\sigma B_{t-M}-a_t \in[M^{1/3}, \, M^{2/3}]\}} \,
    \ee^{- 2 \int_{[0,M]\cup[t-M,t]} h_{t-v}(y-w -\sigma B_v) \d v}
    \Big]
    \\
 &=& \E\Big[ {\bf 1}_{\{\sigma B_{s}\ge -x,\, s\in[0,M]\}} \,
    \ee^{- 2 \int_0^M F_W(z- \sigma B_v) \d v}
    (x+\sigma B_M) \Big]
    {1\over 2(M\pi)^{1/2}} \int_{\r_+} r\kappa_M(r) \d r.
\end{eqnarray*}

\noindent By Lemma \ref{l:bessel decomposition},
$$
\lim_{M\to\infty} \E\Big[{\bf 1}_{\{\sigma B_{s}\ge -x,\,s\in[0,M]\}} \, \ee^{- 2 \int_0^M F_W(z - \sigma B_v) \d v}(x+\sigma B_M)\Big]
=
\varphi_x(z) ,
$$

\noindent with the notation of (\ref{phi}). So it remains to check that
\begin{equation}
    \lim_{M\to \infty}
    {1\over 2(M\pi)^{1/2}} \int_{\r_+} r\kappa_M(r) \d r
    =
    f(w, \, \zeta),
    \label{limitM}
\end{equation}

\noindent for some $f(w, \, \zeta)$ such that $f(w, \, \zeta) \sim |w|$ as $w\to -\infty$ and uniformly in $\zeta>0$.

Recalling the definition of $\kappa_M$ in (\ref{def:kappa}), we have
\begin{eqnarray*}
 &&\int_{\r_+} r\kappa_M(r) \d r
    \\
 &=&\int_{M^{1/3}}^{M^{2/3}}  r\,
    \ee^{-{(z-w+x-r)^2\over 2\sigma^2 M}}
    \E_{{r\over \sigma},{x+z-w\over \sigma}}^{(M)}
    \Big[ {\bf 1}_{\{ \min_{[0,M]} B>0 \}}
    \ee^{- 2  \int_0^M h_{M-v}(x+z-w - \sigma B_v) \d v}\Big] \d r
    \\
 &=&\int_{M^{1/3}}^{M^{2/3}}  r\,
    \ee^{-{(z-w+x-r)^2\over 2\sigma^2 M}}
    \E_{{x+z-w\over \sigma},{r\over \sigma}}^{(M)}
    \Big[ {\bf 1}_{\{ \min_{[0,M]} B>0 \}}
    \ee^{- 2  \int_0^M h_v (x+z-w - \sigma B_v) \d v}\Big] \d r
    \\
 &=& \sigma (2\pi M)^{1/2} \, \E_{{x+z-w\over \sigma}}
    \Big[ \sigma B_M \,
    {\bf 1}_{\{ \sigma B_M \in [M^{1/3}, \, M^{2/3}]\} } \,
    {\bf 1}_{\{ \min_{[0,M]} B>0 \}}
    \ee^{- 2  \int_0^M h_v (x+z-w - \sigma B_v) \d v}\Big] ,
\end{eqnarray*}

\noindent which, by the $h$-transform of the Bessel process, is
$$
=
\sigma (2\pi M)^{1/2}\, (x+z-w)\, \E_{{x+z-w\over \sigma}} \Big[ \ee^{- 2  \int_0^M h_v (x+z-w-\sigma R_v) \d v}{\bf 1}_{\{\sigma R_M \in [M^{1/3}, \, M^{2/3}]\}} \Big].
$$

\noindent Dominated convergence implies that
\begin{eqnarray*}
    \lim_{M\to\infty} {1\over M^{1/2}}
    \int_{\r_+} r\kappa_M(r) \d r
 &=& \sigma (2\pi)^{1/2}\, (x+z-w)\,
    \E_{{x+z-w\over \sigma}}
    \Big[ \ee^{- 2  \int_0^\infty h_v (x+z-w-\sigma R_v) \d v} \Big]
    \\
 &=& \sigma (2\pi)^{1/2}\, (x+z-w) \,
    \E_{{x+z-w\over \sigma}}
    \Big[ \ee^{- 2 \int_0^\infty G_{\zeta+v} (x+z-\sigma R_v) \d v} \Big].
\end{eqnarray*}

\noindent This yields (\ref{limitM}) with
$$
f(w,\zeta)
:=
(x+z-w)\, \E_{{x+z-w\over \sigma}} \Big[\ee^{- 2  \int_0^\infty G_{\zeta+v}(x+z- \sigma R_v) \d v} \Big] ,
$$

\noindent and thus the first part of Lemma \ref{l:g2}. It remains to check that $f(w,\zeta)\sim |w|$ as $w\to-\infty$, uniformly in $\zeta>0$. We only have to show that, uniformly in $\zeta>0$,
$$
\lim_{w\to-\infty} \E_{{x+z-w\over \sigma}} \Big[ \ee^{- 2  \int_0^\infty G_{\zeta+v}(x+z-\sigma R_v) \d v}\Big]
=
1 .
$$
Using again (\ref{eq:stretching}), $G_v(m_v +r )\le c\, (|r|+1) \ee^{r}$ for any $v\ge 0$ and $r\in\r$, we have that
$$
\int_0^\infty G_{\zeta+v}(x+z-R_v) \d v \le \ee^{x+z}\int_0^\infty \ee^{-R_v} \d v ,
$$

\noindent and we conclude by $\lim_{r\to\infty} \E_{r}[\ee^{- c \int_0^\infty \ee^{-R_v} \d v}]=1$ for any fixed $c>0$.
\hfill$\Box$

\medskip

The rest of the section is devoted to the proof of Lemmas \ref{l:conv-kappa} and \ref{l:bessel decomposition}.

\medskip

\noindent{\it Proof of Lemma  \ref{l:conv-kappa}}.  For any $a$, $\eta>0$, we have by (\ref{joint-min-Bt})
$$
\P_a\Big( \min_{[0,t]}\sigma B_s>0,\; \sigma B_t \in \!\d\eta\Big)
=
\Big( {2\over \pi \sigma^2 t}\Big)^{\! 1/2} \, \ee^{-{\sigma^2 a^2+\eta^2\over 2\sigma^2 t}} \sinh \Big({\eta a\over \sigma t}\Big) \d \eta .
$$

\noindent In particular, if ${a\eta\over t}\to 0$ as $t\to\infty$, we have (recalling that $\sigma^2 =2$)
$$
\P_a\Big( \min_{[0,t]}\sigma B_s>0,\; \sigma B_t \in \!\d\eta\Big)
\sim
{1\over \sqrt{2\pi}} \, {a\eta\over t^{3/2}} \, \ee^{-{\sigma^2 a^2+\eta^2\over 2\sigma^2 t}} \d\eta.
$$

\noindent Fix $\eta>0$. By the Markov property at time ${t\over 2}$, and using the fact that ${B_{{t\over 2}}}$ is of order $t^{1/2}$, we have, for $t\to \infty$,
\begin{eqnarray*}
 && \P_b\Big( \{ \sigma B_s\ge a_s, \; s\in[0,t] \}
    \cap
    \{ \sigma B_t\in a_t + \!\d \eta\} \Big)
    \\
 &=& \E_b\Big[
    {\bf 1}_{\{ \sigma B_s\ge -x, \; s\in[0,\, {t\over 2}] \}} \,
    \P_{B_{{t\over 2}}- {a_t\over \sigma} }
    \Big(\min_{[0,\, {t\over 2}]}B_s>0,\,
    \sigma B_{{t\over 2}} \in \!\d \eta \Big) \Big]
    \\
 &\sim& {2\over \sqrt{\pi}}{\eta\over t^{3/2}}
    \E_b\Big[ {\bf 1}_{\{ \sigma B_s\ge -x,\, s\in[0,\, {t\over 2}] \}}
    \, B_{{t\over 2}} \, \ee^{-{B_{t/2}^2\over t}}\Big] \d\eta.
\end{eqnarray*}

\noindent Going from the killed Brownian motion to the three-dimensional Bessel process, we see that, as $t\to\infty$,
$$
\E_b\Big[ {\bf 1}_{\{ \sigma B_s\ge -x,\, s\in[0,\, {t\over 2}] \}}
\, B_{{t\over 2}} \, \ee^{-{B_{t/2}^2\over t}}\Big]
\sim
(b+ {x\over \sigma}) \, \E_0\Big[ \ee^{-{R_{t/2}^2\over t}}\Big]
=
2^{-3/2} (b+ {x\over \sigma}) .
$$

\noindent Hence,
$$
\P_b\Big( \{ \sigma B_s\ge a_s, \; s\in[0,t] \} \cap \{ \sigma B_t\in a_t + \!\d \eta\} \Big)
\sim
{\sigma b + x\over 2\sqrt{\pi}} \, {\eta\over t^{3/2}} \d \eta.
$$

\noindent To complete the proof, we have to use dominated convergence. It is enough to show that (recalling that the function $\kappa$ is bounded with compact support) for any $K>0$,
\begin{equation}
    \label{conv-dominee}
    \sup_{t\ge 1} t^{3/2} \,
    \P_b\Big( \{B_s\ge a_s,s\in[0,t]\} \cap \{B_t-a_t \le K\} \Big)
    <\infty.
\end{equation}

\noindent This can easily be deduced from (\ref{joint-min-Bt}).
\hfill$\Box$

\medskip

\noindent{\it Proof of Lemma  \ref{l:bessel decomposition}.} We have
\begin{eqnarray*}
   &&\E\Big[
     {\bf 1}_{\{\sigma B_{s}\ge -x,\,s\in[0,M]\}}
     \ee^{- 2  \int_0^M F_W(z - \sigma B_v) \d v}
     (x+\sigma B_M)\Big]
    \\
 &=&\E_{{x\over \sigma}}\Big[
    {\bf 1}_{\{\sigma B_{s}\ge 0,\,s\in[0,M]\}}
    \ee^{- 2  \int_0^M F_W(z+x-\sigma B_v) \d v}
    \sigma B_M\Big]
    \\
 &=&x\E_{{x\over \sigma}}\Big[
    \ee^{- 2  \int_0^M F_W(z+x-\sigma R_v) \d v} \Big] ,
\end{eqnarray*}

\noindent giving the first identity by dominated convergence. To prove the second identity, we recall the following well known path decomposition for the three-dimensional Bessel process $R$: under $\P_{{x\over \sigma}}$, $\inf_{s\ge 0} R_s$ is uniformly distributed in $(0, \, {x\over \sigma})$. Furthermore, if we write $\nu := \inf\{ s\ge 0: \, R_\nu = \inf_{s\ge 0} R_s\}$, the location of the minimum, then conditionally on $\inf_{s\ge 0} R_s = r \in (0, \, {x\over \sigma})$, the pre-minimum path $({x\over \sigma} - R_s, \, s\in [0, \, \nu])$ and the post-minimum path $(R_{s+\nu} - r, \, s\ge 0)$ are independent, the first being Brownian motion starting at $0$ and killed when hitting ${x\over \sigma} -r$ for the first time, and the second a three-dimensional Bessel process starting at $0$. Accordingly,
$$
x\E_{{x\over \sigma}}\Big[ \ee^{- 2  \int_0^M F_W(z+x-\sigma R_v) \d v} \Big]
=
x \int_0^{{x\over \sigma}} {\sigma\over x} \d r\, \E\Big[ \ee^{-2 \int_0^{T_{{x\over \sigma}-r}} F_W(z+\sigma B_s) \d s - 2\int_0^\infty F_W(z+x-\sigma r - \sigma R_s) \d s} \Big],
$$

\noindent where, as before, the three-dimensional Bessel process $R$ and the Brownian motion $B$ are assumed to be independent, and $T_b := \inf\{ s\ge 0: \, B_s=b\}$ for $b\in \r$. By a change of variables $b:= {x\over \sigma}-r$, we see that the expression on the right-hand is
$$
=
\sigma \int_0^{{x\over \sigma}}
\E\Big[ \ee^{-2 \int_0^{T_b} F_W(z+\sigma B_s) \d s - 2\int_0^\infty F_W(z+ \sigma b - \sigma R_s) \d s} \Big] \d b ,
$$

\noindent which is $\varphi_x(z)$ in (\ref{phi}).\hfill$\Box$



\section{Proof
of Theorem \ref{t:main}}\label{S:proof of thm main}


In this section, we prove Theorem \ref{t:main}. The key result  is  Theorem \ref{T: structure de Q}. The ingredients needed in addition are Proposition \ref{P:convergence jointe vers P et Z} which explains the appearance of the point measure $\sP$, and Proposition \ref{P: no branching at intermediate times} which shows that particles sampled near $X_1(t)$ either have a very recent common ancestor or have branched at the very beginning of the process.  This last result has been first proved by Arguin et al.\ in \cite{ABK1}.

We employ a very classical approach: we stop the particles when they reach an increasing family of affine stopping lines and then consider their descendants independently. The same kind of argument with the same stopping lines appear in \cite{kyprianou} and in \cite{elie}.

\medskip

Fix $k \ge 1$ and consider  $\sH_k$ the set of all particles which are the first in their line of descent to hit the spatial position $k.$ (For the formalism of particle labelling, see Neveu~\cite{neveu88}.) Under the conditions we work with, we know that almost surely $\sH_k$ is a finite set. The set  $\sH_k$ is again a {\it dissecting stopping line} at which we can apply the  the strong Markov property (see e.g. \cite{chauvin91}). We see that conditionally on $\mathscr{F}_{\!\!\mathscr{H}_k} $ 
--- the sigma-algebra generated by the branching Brownian motion when the particles are stopped upon hitting the position $k$ --- the subtrees rooted at the points of $\sH_k$ are independent copies of the
branching Brownian motion
started at position $k$ and at the random time
at which the particle considered has hit $k$. Define $H_k := \# \mathscr{H}_k$ and 
$$
Z_k:= k \ee^{- k}  H_k.
$$

Neveu (\cite{neveu88}, equation (5.4)) shows that the limit $Z$ of the derivative martingale in  (\ref{E:def de Z}) can also be obtained as a limit of $Z_k$ (it is the same martingale on a different stopping line)
\begin{equation}
    Z
    = \lim_{k \to \infty } Z_k =
    \lim_{k\to \infty} k \ee^{ -k} H_k 
    \label{neveu}
\end{equation}
almost surely. 
Let us further define $\sH_{k,t}$ as the set of all particles which are the first in their line of descent to hit the spatial position $k$, and which do so before time $t$.


For each  $u\in \mathscr{H}_{k,t}$, let us write $X_1^u (t)$ for the minimal position at time $t$ of the particles which are descendants of $u$. If $u \in \sH_k \backslash \mathscr{H}_{k,t}$ we define $X_1^u(t) =0.$ This allows us to define the point measure
$$
 \mathscr{P}^*_{k,t} := \sum_{ u\in \mathscr{H}_k} \delta_{X_1^u (t) - \, m_t + \log (CZ_k)}.
$$

We further define 
$$
\sP^*_{k,\infty}
:=
\sum_{u \in \sH_k} \delta_{ k+ W(u) + \log (CZ_k)}
$$
where, conditionally on $\sF_{\sH_k}$, the $W(u)$ are
independent copies
of the
random variable $W$
in (\ref{bramson}). 
\begin{proposition}\label{P:convergence jointe vers P et Z}
The following convergences hold in distribution
$$
\lim_{t \to \infty} \sP^*_{k,t}
=
\sP^*_{k,\infty}
$$
and
$$
\lim_{k \to \infty} (  \mathscr{P}^*_{k,\infty} , Z_k) = (\sP, Z)
$$
where $\sP$ is as in Theorem \ref{t:main}, $Z$ is as in (\ref{E:def de Z1}), and $\sP$ and $Z$ are independent.
\end{proposition}
\begin{proof}

Fix $k \ge 1$. Recall that  $\sH_k$ is the set of particles absorbed at level $k$, and $H_k=\#\sH_k$. 
Observe that for each  $u\in \mathscr{H}_k$,  $X_1^u (t)$ has the same distribution as  $k+X_1(t- \xi_{k,u})$, where $\xi_{k,u}$ is the random time at which $u$ reaches $k$. By (\ref{bramson}) and the fact that $m_{t+c}-m_t \to 0$ for any $c$, we have, for all  $k\ge 1$ and all $u\in \mathscr{H}_k$,
$$
X_1^u (t) -  m_t
\;\; {\buildrel law \over \to} \;\;
k+  W, \qquad  t \to \infty .
$$

\noindent Hence, the finite point measure $\mathscr{P}_{k,t} := \sum_{u \in \sH_k} \delta_{ X_1^u (t) -  m_t}$ converges in distribution as  $t\to \infty$, to $\mathscr{P}_{k,\infty}:= \sum_{u \in \sH_k} \delta_{ k+W(u)}$,
where conditionally on  $\mathscr{H}_k$, the $W(u)$ are
independent copies
of $W$. This proves the first part of Proposition \ref{P:convergence jointe vers P et Z}.

\bigskip

The proof of the second part relies on some classical
extreme value theory.
We refer the reader to \cite{resnick} for a thorough treatment of this subject. Let us state the result we will use. Suppose we are given a sequence $(X_i, i \in \N)$ of i.i.d.\ random variables such that
$$
\P(X_i \ge x) \sim Cx \ee^{- x}, \text{ as } x \to \infty.
$$
Call $M_n = \max_{i=1,\ldots, n} X_i$ the record of the $X_i$. Then it is not hard to see that if we let $b_n = \log n + \log \log n$ we have as $n\to\infty$
\begin{eqnarray*}
\P(M_n -b_n \le y ) &= & (\P(X_i \le y+b_n))^n \\
&= & (1-(1+o(1))C (y+b_n)\ee^{-(y+b_n)})^n \\
&\sim&  \exp \Big( -nC (y+ b_n) \frac1{n\log n} \ee^{- y} \Big) \\
&\sim& \exp(-C\ee^{- y})
\end{eqnarray*}
and therefore
\begin{align*}
\P \left( M_n -b_n - \log C \le y \right)
&\sim \exp(-\ee^{-y}).
\end{align*}

By applying Corollary 4.19 in \cite{resnick} we immediately see that the point measure
$$
\zeta_n := \sum_{i=1}^n \delta_{X_i - b_n - \log C  }
$$
converges in distribution to a Poisson point measure on $\r$ with intensity $\ee^{-x}\d x.$

This result
applies immediately to the
random variables $-W(u)$ (recalling from (\ref{lalley-sellke}) that $\P( -W \ge x) \sim C x \ee^{-x}$, $x\to \infty$) and thus the point measure
$$
\sum_{u \in \sH_k} \delta_{W(u) + ( \log H_k +\log \log H_k + \log C) }
$$
converges
(as $k \to  \infty$) in distribution towards a Poisson point measure on $\r$ with intensity $\ee^{ x }\d x$ (it is $\ee^x$ instead of $\ee^{-x}$ because we are looking at the leftmost particles) independently of $Z$ (this identity comes from (\ref{neveu})). By definition $H_k =  k^{-1} \ee^{ k } Z_k$, thus
\begin{align*}
\log H_k &=   k +\log Z_k - \log k  \\
\log \log H_k &= \log k +\log (1+o_k(1))
\end{align*}
where the term $o_k(1)$ tends to 0 almost surely when $k\to \infty$.
Hence,
$$
\log H_k + \log \log H_k =  \log Z_k +  k + o_k(1).
$$
We conclude that for $u \in \sH_k$
\begin{align*}
k+  W(u) +  \log (C Z) &=  W(u) +  (\log H_k + \log \log H_k +\log C ) + o_k(1).
\end{align*}
Hence we conclude that
$$
\sP^*_{k,\infty}= \sum_{u \in \sH_k} \delta_{k+  W(u) +  \log (C Z)}
$$
also converges (as $k \to  \infty$) towards a Poisson point measure on $\r$ with intensity $\ee^{ x}\d x$ independently of $Z = \lim_k Z_k.$ This concludes the proof of Proposition \ref{P:convergence jointe vers P et Z}.
\end{proof}

\medskip


Recall that $J_\eta(t) := \{ i \le N(t) : |X_i(t)-m_t |\le \eta \}$ is the set of indices which correspond to particles  near $m_t$ at time $t$   and that  $\tau_{i,j}(t)$  is the time at which the particles $X_i(t)$ and $X_j(t)$ have branched from one another. 

\begin{proposition}\label{P: no branching at intermediate times}
{\rm (Arguin, Bovier and Kistler~\cite{ABK1})}
Fix $\eta >0$ and any function $\zeta : [0,\infty ) \to [0,\infty) $ which increases to infinity. Define the event
$$
\mathcal B_{\eta,k,t}:= \left\{ \exists i,j \in J_\eta(t) : \tau_{i,j}(t) \in [\zeta(k) , t-\zeta(k)] \right\}.
$$
One has
\begin{equation}
\lim_{k \to \infty} \lim_{t \to \infty} \P \left[  \mathcal B_{\eta,k,t}  \right] =0.
\end{equation}
\end{proposition}
The following proof is included for the sake of self-containedness. 
\begin{proof}

Fix $\eta>0$ and $k\to\zeta(k)$ an increasing sequence going to infinity. 
We want to control the probability  of
$$
\mathcal B_{\eta,k,t}= \left\{ \exists i,j \in J_\eta(t) : \tau_{i,j}(t) \in [\zeta(k) , t-\zeta(k)] \right\}
$$
the ``bad" event that particles have branched at an intermediate time when $t \to \infty$ and then $k \to \infty.$

By choosing $x$ large enough, we have for all $\zeta\ge 0$ and $t$ large enough
\begin{align*}
& \P(\exists i,j \in J_\eta(t) : \tau_{i,j}(t) \in [\zeta, t-\zeta] ) \\ &\le \P(A_t(x,\eta)^\complement )+\P(\exists i,j \in J_\eta(t) : \tau_{i,j}(t) \in [\zeta, t-\zeta] , A_t(x,\eta))  \\
&\le \varepsilon + \E\left[ {\bf 1}_{A_t(x,\eta)}\sum_{i \in J_\eta(t) } \mathbf{1}_{\{ \exists j \in J_\eta(t) : \tau_{i,j}(t) \in [\zeta, t-\zeta]  \} }  \right].  
\end{align*}
 Using the many-to-one principle (see (\ref{E:many to one version 1})), we have
\begin{align*}
  \E\left[ {\bf 1}_{A_t(x,\eta)}\sum_{i \in J_\eta(t) } \mathbf{1}_{\{ \exists j \in J_\eta(t) : \tau_{i,j}(t) \in [\zeta, t-\zeta]  \} }  \right]
 &= \E_{\Q}\Big[\ee^{X_{\Xi_t}(t)}{\bf 1}_{A_t(x,\eta)}{\bf 1}_{\{|X_{\Xi_t}(t) -m_t | \le \eta , \exists j \in J_\eta(t) : \tau_{\Xi,j}(t) \in [\zeta, t-\zeta]\}} \Big] 
\end{align*}
where  $\tau_{\Xi,j}(t)$ is the time at which the particle $X_j(t)$ has branched off the spine $\Xi$. In particular, using the description of the process under $\Q$, we know that $X_{\Xi_t}(t)$ is  $\sigma$ times a standard Brownian motion, and that independent branching Brownian motions are born at rate 2 (at times $(\tau_{i}^{(\Xi_t)}(t),i\ge 1)$) from the spine $\Xi$. The event $\{ \exists j \in J_\eta(t) : \tau_{\Xi,j}(t) \in [\zeta, t-\zeta]\}\}$ means that there is an instant $\tau_{i}^{(\Xi_t)}(t)$ between $\zeta$ and $t-\zeta$, such that the branching Brownian motion that separated from $\Xi$ at that time has a descendant at time $t$ in $[m_t-\eta,m_t+\eta]$. In particular, the minimum of this branching Brownian motion at time $t$ is lower than $m_t+\eta$. Thus
\begin{align*}
& \E_\Q\Big[\ee^{X_{\Xi_t}(t)}{\bf 1}_{A_t(x,\eta)}{\bf 1}_{\{|X_{\Xi_t}(t) -m_t | \le \eta , \exists j \in J_\eta(t) : \tau_{\Xi,j}(t) \in [\zeta, t-\zeta]\}} \Big]
\\ &\qquad  \le
\E_\Q\Big[\ee^{X_{\Xi_t}(t)}{\bf 1}_{A_t(x,\eta)}{\bf 1}_{\{|X_{\Xi_t}(t) -m_t | \le \eta \}} \sum_{\tau \in [\zeta, t-\zeta]}   \mathbf{1}_{\{ X_{1,t}^{\tau} \le m_t +\eta \}} \Big] 
\end{align*}
where $ X_{1,t}^{\tau}$ is the leftmost particle at time $t$ descended from the particle which branched off $\Xi$ at time $\tau$, and the sum goes over all times $\tau=\tau_{i}^{(\Xi_t)}(t) \in[\zeta,t-\zeta]$ at which a new particle is created. Recall that $G_v(x) = \P(X_1(v)\le x)$ so that by conditioning we obtain
\begin{align*}
& \E_\Q\Big[\ee^{X_{\Xi_t}(t)}{\bf 1}_{A_t(x,\eta)}{\bf 1}_{\{|X_{\Xi_t}(t) -m_t | \le \eta , \exists j \in J_\eta(t) : \tau_{\Xi,j}(t) \in [\zeta, t-\zeta]\}} \Big]
\\
&\qquad \le \E_\Q\Big[\ee^{X_{\Xi_t}(t)}{\bf 1}_{A_t(x,\eta)}{\bf 1}_{\{|X_{\Xi_t}(t) -m_t | \le \eta \}} \sum_{\tau \in [\zeta, t-\zeta]}  G_{t-\tau}(m_t+\eta-X_{\Xi_\tau}(\tau) )\Big].
\end{align*}

For all continuous function $X:[0,t] \to \r$ recall that we define $\underline X^{[a,b]}:= \min_{s\in[a,b]} X(s)$, and define
the event $A_t^{(X)}(x,\eta)$ by
\begin{align*}
A_t^{(X)}(x,\eta) &:= \{ \sigma \underline X^{[0,t/2]} \ge -x   \} \cap \{ \underline X^{[t/2, t]} \ge m_t -x   \} \cap \{ \forall s \in [x,t/2] : X(s) \ge s^{1/3}  \} \\ &\qquad \cap \{   \forall s \in [t/2, t-x] : X(s) -X(t) \in  [(t-s)^{1/3},(t-s)^{2/3}] \} .
\end{align*}
Then, $A_t(x,\eta)\cap\{|X_{\Xi_t}(t)-m_t|<\eta\}\subset A_t^{(X_{\Xi_t,t}(\cdot))}(x,\eta)\cap \{|X_{\Xi_t}(t)-m_t|<\eta\}$, hence
\begin{align*}
& \E_\Q\Big[\ee^{X_{\Xi_t}(t)}{\bf 1}_{A_t(x,\eta)}{\bf 1}_{\{|X_{\Xi_t}(t) -m_t | \le \eta \}} \sum_{\tau \in [\zeta, t-\zeta]}   \mathbf{1}_{\{ X_{1,t}^{\tau} \le m_t +\eta \}} \Big] \\
&\le \E_\Q\Big[\ee^{X_{\Xi_t}(t)}{\bf 1}_{A_t^{(X_{\Xi_t,t}(\cdot))}(x,\eta)}{\bf 1}_{\{|X_{\Xi_t}(t) -m_t | \le \eta \}} \sum_{\tau \in [\zeta, t-\zeta]}  G_{t-\tau}(m_t+\eta-X_{\Xi_\tau}(\tau) )\Big]
\\ 
&= \E\Big[\ee^{\sigma B(t)}{\bf 1}_{A_t^{(\sigma B(\cdot))}(x,\eta)}{\bf 1}_{\{|\sigma B(t) -m_t | \le \eta \}} \sum_{\tau \in [\zeta, t-\zeta]}  G_{t-\tau}(m_t+\eta-\sigma B(\tau) )\Big]
\end{align*}
where in the last expectation $B$ is a standard Brownian motion and the $\tau$ over which the sums run are the atoms of a rate 2 Poisson process independent of $B$.
Since we are on the good event $A_t^{(\sigma B)}(x,\eta)$, we know that for $x\le s \le t/2, \sigma B_s >s^{1/3} $ and $\sigma B_{t-s} > m_t + s^{1/3}.$  Therefore
\begin{align*}
&\E_\Q\Big[\ee^{\sigma B_t}{\bf 1}_{A_t^{(\sigma B)}(x,\eta)}{\bf 1}_{\{|\sigma B_t -m_t | \le \eta \}} \sum_{\tau \in [\zeta, t-\zeta]}  G_{t-\tau}(m_t+\eta-\sigma B(\tau) ) \Big] 
\\   &\le  t^{3/2}\ee^{\eta+C_B} \P\left(A_t^{(\sigma B)}(x,\eta),|\sigma B_t -m_t | \le \eta \right) \Big\{\int_\zeta^{t/2} 2  G_{t-s}( m_t -s^{1/3} +\eta) \d s \\ & \qquad \qquad \qquad \qquad \qquad \qquad \qquad \qquad+  \int_\zeta^{t/2} 2  G_s(  -s^{1/3} +\eta) \d s \Big\}\\
& \le c\, \Big\{ \int_\zeta^{t/2} 2  G_{t-s}( m_t -s^{1/3} +\eta) \d s + \int_\zeta^{t/2} 2  G_s(-s^{1/3} +\eta) \d s \Big\}
\end{align*}

\noindent where the constant $c$ only depends on $\eta$ and where we have used (\ref{joint-min-Bt}) for the last inequality.

Now, observe that $G_{t-s}( m_t -s^{1/3} +\eta) = G_{t-s}(m_{t-s}({1\over 2}) + \Delta(\eta,t,s) )$ where, as before, $m_{t-s}({1\over 2})$ is such that $G_{t-s}(m_{t-s}({1\over 2})) = {1\over 2}$, and $\Delta(\eta, t,s) := \eta -s^{1/3}  + m_t -m_{t-s}({1\over 2})$. Since $m_t ={3\over 2}\log t + C_B$ by definition, and $m_{t-s}({1\over 2}) = {3\over 2}\log (t-s) + C + o(1)$, $t-s \to \infty$, for some constant $C\in \r$, we see that there exists a sufficiently large $\zeta_0$ such that $\Delta(\eta, t,s) \le -{1\over 2} s^{1/3}$, $\forall t>\zeta \ge \zeta_0$, $\forall s\in [\zeta, {t\over 2}]$. This implies, for $t>\zeta \ge \zeta_0$ and $s\in [\zeta, {t\over 2}]$,
$$
G_{t-s}( m_t -s^{1/3} +\eta) 
=
G_{t-s}( m_{t-s}({1\over 2}) + \Delta(\eta, t,s))
\le
\P(W \le \eta -{1\over 2}s^{1/3}),
$$

\noindent the last inequality being a consequence of Fact \ref{f:stretching-lemma}.

Since $\P(W \le -y)\sim cy\ee^{-y}$, $y\to \infty$, we conclude that $\int_\zeta^{t/2} 2 G_{t-s}( m_t -s^{1/3} +\eta) \d s \to 0$ as $\zeta \to \infty.$ A similar argument also shows that $\int_\zeta^{t/2} 2  G_{s}( -s^{1/3} +\eta) \d s \to 0$ as $\zeta \to \infty$.

\medskip

The conclusion here is that by choosing $\zeta$ large enough (depending only on $\eta$),  we have $\P(\exists i,j \in J_\eta(t) : \tau_{i,j}(t) \in [\zeta, t-\zeta] ) <\varepsilon$ uniformly in $t.$ 
\end{proof}


Recall that $\forall u \in \sH_k, X_1^u(t)$ is the position at time $t$ of the leftmost descendent of $u$ (or 0 if $u \not \in \sH_{k,t}$), and let $X_{1,t}^u(s), s\le t$ be the position at time $s$ of the ancestor of this leftmost descendent (or 0 if $u \not \in \sH_{k,t}$).
For each $t,\zeta$ and $u \in \sH_k$ define
$$
\sQ^{(u)}_{t,\zeta} = \delta_0+ \sum_{i : \tau_i^{u}>t-\zeta} \sN_i^{u}
$$
where the $\tau_i^u$ are the branching times along the path $s  \mapsto X_{1,t}^u(s) $ enumerated backward from $t$ and the $\sN_i^{u}$ are the point measures of particles whose ancestor was born at $\tau_i^u$ (this measure has no mass if $u \not \in \sH_{k,t}$). Thus, $\sQ^{(u)}_{t,\zeta}$ is the point measure of particles which have branched off the path $s  \mapsto X_{1,t}^u(s) $ at a time which is posterior to $t-\zeta$, including the particle at $X_1^u(t)$.

In the same manner we define $\sQ_\zeta$ as the point measure obtained from $\sQ$ (in Theorem \ref{T: structure de Q}) by only keeping the particles that have branched off $s \mapsto Y(s)$ before $\zeta.$
More precisely,  conditionally on the path $Y,$ we let $\pi$ be a Poisson point process on $[0,\infty)$ with intensity $2  \big(1 - G_t(-Y(t))\big) \d t  = 2  \big(1 - \P_{Y(t)} (X_1(t) <0) )\big) \d t$. For each point $t\in \pi$ such that $t<\zeta,$ start an independent
branching Brownian motion
$(\sN^*_{Y(t)}(u), u\ge 0)$ at position $Y(t)$ conditioned to have $\min \sN^*_{Y(t)}(t) >0$. Then define $\sQ_\zeta:= \delta_0+ \sum_{t \in \pi, t<\zeta} \sN^*_{Y(t)}(t).$

\begin{lemma}\label{L: joint P and Q}
For each fixed $k$ and $\zeta$, the following limit holds in distribution
$$
 \lim_{t \to \infty} ( \mathscr{P}^*_{k,t}, (\sQ^{(u)}_{t,\zeta} )_{u \in \sH_k} ) = ( \sP^*_{k,\infty} , (\sQ^{(u)}_{\zeta})_{u \in \sH_k} )
$$
where $(\sQ^{(u)}_{\zeta})_{u \in \sH_k} $ is a collection of
independent copies
of $\sQ_{\zeta}$, independent of $\sP^*_{k,\infty}$.
\end{lemma}

\begin{proof}
Conditionally on $\sH_k$,
the random variables
$(X_{1,t}^u(\cdot), \, \sQ^{(u)}_{t,\zeta})_{u\in\sH_k}$ are independent by the branching
property. By Theorem \ref{T: structure de Q}, for every $u\in\sH_k$,
the pair
$(X_1^u(t)-  m_t, \, \sQ^{(u)}_{t,\zeta})$ converges
in law to $(k+  W(u), \, \sQ^{(u)}_{\zeta})$ where
$\sQ^{(u)}_{\zeta}$ is a copy of $\sQ_{\zeta}$ independent of $W(u)$.

To conclude, observe that  $\sum_{u\in \sH_k}\delta_{k +   W(u)} = \sP^*_{k,\infty} -\log (CZ_k)$  by Proposition
\ref{P:convergence jointe vers P et Z}. Since for each $u \in \sH_k$ the point measure $\sQ_\zeta^{(u)}$ is independent of $W(u)$ and of all $W(v)$ for $v \in \sH_k$ and $v\neq u$, it follows that  $\sQ_\zeta^{(u)}$ is independent of  $\sP^*_{k,\infty}$. We conclude that
$$
\lim_{t \to \infty} ( \mathscr{P}^*_{k,t}, (\sQ^{(u)}_{t,\zeta} )_{u
\in \sH_k} ) = ( \sP^*_{k,\infty} , (\sQ^{(u)}_{\zeta})_{u \in \sH_k} )
$$
in distribution where the two components of the limit are independent.
\end{proof}

\medskip

Armed with these tools let us proceed to give the
\begin{proof}[Proof of Theorem \ref{t:main}]
 Let $\bar \sN^{(k)}(t)$ be the extremal point measure seen from the position $m_t-\log (CZ_k)$
$$
\bar \sN^{(k)}(t):= \sN(t)-m_t +\log (CZ_k).
$$

\noindent Let $\zeta : [0,\infty) \to [0,\infty)$ be any function increasing to infinity. Observe that on
$\mathcal B_{\eta,k,t}^\complement$ (an event of probability tending to one when $t\to  \infty$ and then $k \to \infty$ by
Proposition
\ref{P: no branching at intermediate times})
we have
$$
\bar \sN^{(k)}(t) |_{[-\eta , \eta]} = \sum_{ u \in \sH_k} \left( \sQ^{(u)}_{t,\zeta(k)}+  X_{1,t}^u-m_t+\log (CZ_k)\right)|_{[-\eta, \eta) ]}.
$$
Now by Lemma \ref{L: joint P and Q} we know that in distribution $$
\lim_{t \to \infty}  \sum_{ u \in \sH_k} \Big( \sQ^{(u)}_{t,\zeta(k)}+ X_{1,t}^u-m_t + \log (CZ_k) \Big)  = \sum_{x \in \sP^*_{k,\infty} } ( x+ \sQ^{(x)}_{\zeta(k)})
$$
where the $\sQ^{(x)}_{\zeta(k)}$ are
independent copies
of $\sQ_{\zeta(k)}$, and independent of $H_k$. Moreover, we know that $\lim_{t\to \infty} Z(t)=Z$ almost surely.

By the second limit in Proposition \ref{P:convergence jointe vers P et Z}, we have that $( \sum_{x \in \sP^*_{k,\infty} } ( x+ \sQ^{(x)}_{\zeta(k)} ),Z_k)$  converges as $k\to\infty$ to $(\sL,Z)$ in distribution, $\sL$ being independent of $Z$. In particular, $(\sL,Z)$ is also  the limit in distribution of $( \sum_{x \in \sP^*_{k,\infty} } ( x+ \sQ^{(x)}_{\zeta(k)} ),Z)$. We conclude that in distribution
$$
\lim_{k\to\infty} \lim_{t \to \infty} (\bar \sN^{(k)}(t) |_{[-\eta, \eta) ]},Z(t))=  (\sL  |_{[-\eta, \eta) ]},Z).
$$
Hence, $\lim_{k\to\infty}\lim_{t\to\infty}(\bar \sN^{(k)}(t),Z(t))=(\sL,Z)$ in distribution. Since $\bar \sN(t)$ is obtained  from $\bar \sN^{(k)}(t)$ by the shift $\log (CZ) - \log (CZ_k)$, which goes to $0$ by (\ref{neveu}), we have in distribution
$
\lim_{t\to\infty}  (\bar \sN (t),Z(t)) = (\sL,Z) 
$
which yields the content of Theorem \ref{t:main}.
\end{proof}

\medskip

\noindent{\bf Acknowledgements:} We wish to warmly thank two anonymous referees for their  careful reading and fruitful suggestions. We also wish to express our gratitude to Louis-Pierre Arguin for a useful conversation and Henri Berestycki for the argument behind Remark \ref{rmk 6.3}.





\end{document}